\documentclass[11pt]{amsart}
\usepackage{amsfonts}

\theoremstyle{plain}

\newtheorem{corollary}{Corollary}

\newtheorem{lemma}{Lemma}

\newtheorem{proposition}{Proposition}
\newtheorem{remark}{Remark}

\numberwithin{equation}{section}
\input{tcilatex}

\begin{document}
\title[Strong Euler approximation]{On the rate of convergence of strong
Euler approximation for SDEs driven by Levy processes}
\author{R. Mikulevi\v{c}ius and Fanhui Xu}
\address{University of Southern California, Los Angeles}
\date{June 22, 2016}

\begin{abstract}
A SDE driven by an $\alpha $-stable process, $\alpha \in \lbrack 1,2),$ with
Lipshitz continuous coefficient and $\beta $-H\"{o}lder drift is considered.
The existence and uniqueness of a strong solution is proved when $\beta
>1-\alpha /2$ by showing that it is $L_{p}$-limit of Euler approximations.
The $L_{p}$-error (rate of convergence) is obtained for a nondegenerate
truncated and nontruncated driving process. The rate in the case of Lipshitz
continuous coefficients is derived as well.
\end{abstract}

\maketitle
\tableofcontents

\section{Introduction}

Let $\left( \varOmega,\mathcal{F},\mathbf{P}\right) $ be a complete
probability space, and $\mathbb{F}=\left( \mathcal{F}_{t}\right) _{t\in
\lbrack 0,1)}$ be a filtration of $\sigma $-algebras satisfying the usual
conditions. Let $N\left( dt,dy\right) $ be adapted Poisson point measure on $%
[0,1)\times \mathbf{R}_{0}^{d}$ ($\mathbf{R}_{0}^{d}=\mathbf{R}%
^{d}\backslash \left\{ 0\right\} $) such that%
\begin{equation*}
\mathbf{E}N\left( dt,dy\right) =\rho \left( y\right) \frac{dydt}{\left\vert
y\right\vert ^{d+\alpha }},
\end{equation*}%
where $\rho \left( y\right) $ is a bounded measurable function, and $\alpha
\in \lbrack 1,2)$. We consider the following stochastic differential
equation (SDE) in time interval $\left[ 0,1\right) $%
\begin{equation}
X_{t}=x_{0}+\int_{0}^{t}b\left( X_{s}\right) ds+\int_{0}^{t}G\left(
X_{s-}\right) dL_{s}.  \label{m1}
\end{equation}%
The drift coefficient $b:\mathbf{R}^{d}\longrightarrow \mathbf{R}^{d}$ is a
bounded function of $\beta $-H\"{o}lder continuity in whole space with $%
\beta \in \left( 0,1\right] $, $G\left( x\right) ,x\in \mathbf{R}^{d},$ is a
Lipshitz continuous bounded $d\times d$ -matrix$,$ and for $t\in \lbrack
0,1),$ 
\begin{eqnarray*}
L_{t} &=&\int_{0}^{t}\int yq\left( ds,dy\right) ,\text{ if }\alpha \in
\left( 1,2\right) , \\
L_{t} &=&\int_{0}^{t}\int_{\left\vert y\right\vert >1}yN\left( ds,dy\right)
+\int_{0}^{t}\int_{\left\vert y\right\vert \leq 1}yq(ds,dy),\text{ if }%
\alpha =1,
\end{eqnarray*}%
where 
\begin{equation*}
q\left( dt,dy\right) =N\left( dt,dy\right) -\rho \left( y\right) \frac{dydt}{%
\left\vert y\right\vert ^{d+\alpha }}
\end{equation*}%
is a martingale measure. We will need the following assumptions for $\rho $.

\textbf{S}$(c_{0}$). \emph{(i)~}$\rho \left( y\right) \geq c_{0},y\in 
\mathbf{R}_{0}^{d}$\emph{\ for some }$c_{0}>0$\emph{;\newline
(ii)\thinspace ~}$\rho \left( \lambda y\right) =\rho \left( y\right) $\emph{%
\ for all }$\lambda >0,$\emph{\ }$y\in \mathbf{R}_{0}^{d}$\emph{, i.e., }$%
\rho $\emph{\ is a }$0$\emph{-homogeneous function;\newline
(iii) } 
\begin{equation}
\rho \left( -y\right) =\rho \left( y\right) ,y\in \mathbf{R}_{0}^{d},\text{
if }\alpha =1.  \label{sym}
\end{equation}

We are going to study the Euler approximation to (\ref{m1}) defined as 
\begin{equation}
X_{t}^{n}=x_{0}+\int_{0}^{t}b\left( X_{\pi _{n}\left( s\right) }^{n}\right)
ds+\int_{0}^{t}G\left( X_{\pi _{n}\left( s\right) }^{n}\right) dL_{s},
\label{m2}
\end{equation}%
where $\pi _{n}\left( s\right) =k/n$ if $k/n\leq s<(k+1)/n,n=1,2,\ldots
,k=0,\ldots ,n-1.$ Note that the driving process $L_{t}$ does not have $%
\alpha $-moment.

Sometimes in (\ref{m1}) $L_{t}$ is replaced by its truncation%
\begin{equation*}
L_{t}^{0}=\int_{0}^{t}\int_{\left\vert y\right\vert \leq 1}yq\left(
dr,dy\right) ,t\in \lbrack 0,1),
\end{equation*}%
i.e., the following equation and the accompanying Euler approximation are
considered instead, 
\begin{equation}
Y_{t}=x_{0}+\int_{0}^{t}b\left( Y_{s}\right) ds+\int_{0}^{t}G\left(
Y_{s-}\right) dL_{s}^{0},t\in \lbrack 0,1),  \label{m1'}
\end{equation}%
and%
\begin{equation}
Y_{t}^{n}=x_{0}+\int_{0}^{t}b\left( Y_{\pi _{n}\left( s\right) }^{n}\right)
ds+\int_{0}^{t}G\left( Y_{\pi _{n}\left( s\right) }^{n}\right)
dL_{s}^{0},t\in \lbrack 0,1).  \label{m2'}
\end{equation}%
This case would be the other concern of our note. It is well-known that the
truncated driving process $L_{t}^{0}$ has all moments.

In \cite{p}, the existence and uniqueness of strong solutions to (\ref{m1})
was considered by assuming $G=I_{d}$, the $d\times d$-identity matrix, and
with $L_{t}$ being nondegenerate $\alpha $-stable symmetric, $\alpha \in
\lbrack 1,2),\beta >1-\alpha /2$. The pathwise uniqueness for (\ref{m1}) was
proved by applying Gronwall's lemma and using the elliptic version of the
Kolmogorov equation and regularity of its solution, to represent the H\"{o}%
lder drift $b\left( x\right) $ by an expression which is \textquotedblleft
Lipshitz\textquotedblright. This approach, \textquotedblleft It\^{o}-Tanaka
trick\textquotedblright, was inspired by considerations in \cite{fp}, see
the infinite dimensional generalization in \cite{dp} for $G=I$ and $L=W$
being Wiener, or a finite dimensional generalization (using parabolic
backward Kolmogorov equations) in \cite{ff}, again with $G=I_{d},L=W $, and $%
b$ having some integrability properties.

On the other hand, in \cite{ta} a truncated equation (\ref{m1'}) and its
Euler approximation (\ref{m2'}) were considered with $G=I_{d},\rho =1$.
Using the same It\^{o}-Tanaka trick and assuming that a strong solution $%
Y_{t}$ exists with $\alpha +\beta >2,\beta \in \left( 0,1\right) $, the rate
of strong convergence was derived. It was proved in \cite{ta} that 
\begin{equation}
\mathbf{E}\left[ \sup_{t}\left\vert Y_{t}^{n}-Y_{t}\right\vert ^{p}\right]
\leq C_{p}\left\{ 
\begin{array}{cc}
n^{-1} & \text{if }p\geq 2/\beta , \\ 
n^{-p\beta /2} & \text{if }2\leq p<2/\beta .%
\end{array}%
\right.  \label{11}
\end{equation}

In this note, using It\^{o}-Tanaka trick again, we derive the rate of
convergence of Euler approximations for both (\ref{m1}) and (\ref{m1'}). We
show that, under the imposed assumptions, $X^{n},Y^{n}$ are Cauchy sequences
whose limits solve (\ref{m1}) and (\ref{m1'}) respectively.

For (\ref{m1}), the following holds. Note that only the moments $p<\alpha $
exist in this case.

\begin{proposition}
\label{pro2}Let $\alpha \in \lbrack 1,2),\mathbf{S}\left( c_{0}\right) $
hold, $\beta \in (0,1)$ and $\beta >1-\alpha /2$. Assume $b\in C^{\beta
}\left( \mathbf{R}^{d}\right) ,$ $G$ is bounded Lipshitz and $\left\vert
\det G\left( x\right) \right\vert \geq c_{0}>0,x\in \mathbf{R}^{d}$, i.e. $G$
is uniformly nondegenerate. Let for some $c_{1}>0,$ 
\begin{equation*}
\left\vert \rho \left( y\right) -\rho \left( z\right) \right\vert \leq
c_{1}\left\vert y-z\right\vert ^{\beta }\text{ for all }\left\vert
y\right\vert =\left\vert z\right\vert =1.
\end{equation*}
Then there is a unique strong solution to (\ref{m1})$.$ Moreover for each $%
p\in \left( 0,\alpha \right) $, there is $C$ depending on $d,\alpha ,\beta
,b,G,p,\rho $ such that%
\begin{equation*}
\mathbf{E}\left[ \sup_{t}\left\vert X_{t}^{n}-X_{t}\right\vert ^{p}\right]
\leq Cn^{-p\beta /\alpha }.
\end{equation*}
\end{proposition}

For (\ref{m1'}) we derive the following statement which extends and improves
the results in (\cite{ta}), see (\ref{11}).

\begin{proposition}
\label{t1}Let $\alpha \in \lbrack 1,2),\mathbf{S}\left( c_{0}\right) $ hold, 
$\beta \in (0,1)$ and $\beta >1-\alpha /2$. Assume $b\in C^{\beta }\left( 
\mathbf{R}^{d}\right) ,$ $G$ is bounded Lipshitz and $\left\vert \det
G\left( x\right) \right\vert \geq c_{0}>0,x\in \mathbf{R}^{d}$, i.e. $G$ is
uniformly nondegenerate. Let for some $c_{1}>0,$ 
\begin{equation*}
\left\vert \rho \left( y\right) -\rho \left( z\right) \right\vert \leq
c_{1}\left\vert y-z\right\vert ^{\beta }\text{ for all }\left\vert
y\right\vert =\left\vert z\right\vert =1.
\end{equation*}
Then there is a unique strong solution to (\ref{m1'})$.$ Moreover for each $%
p\in \left( 0,\infty \right) $, there is $C$ depending on $d,\alpha ,\beta
,b,G,p,\rho $ such that%
\begin{equation*}
\mathbf{E}\left[ \sup_{t}\left\vert Y_{t}^{n}-Y_{t}\right\vert ^{p}\right]
\leq C\left\{ 
\begin{array}{ll}
n^{-p\beta /\alpha } & \text{if }0<p<\alpha /\beta , \\ 
\left( n/\ln n\right) ^{-1} & \text{if }p=\alpha /\beta , \\ 
n^{-1} & \text{if }p>\alpha /\beta .%
\end{array}%
\right.
\end{equation*}
\end{proposition}

In both statements above, $L$ and $G$ are nondegenerate (Assumption \textbf{S%
}$\left( c_{0}\right) $ holds). On the other hand, if $b$ and $G$ are
Lipshitz continuous, then there exists a unique solution to \eqref{m1} (see
Theorem 6.2.3, \cite{da}) with any bounded nonnegative $\rho $. In this
note, we use direct estimates of stochastic integrals to derive the
convergence rate in the Lipshitz, possibly completely degenerate, case.

The following statement holds for all Lipshitz case of (\ref{m1}).

\begin{proposition}
\label{pro3}Let $\alpha \in \lbrack 1,2),\rho $ be nonnegative bounded.
Assume $b$ and $G$ are bounded Lipshitz functions. Then

(i) For each $p\in \left( 0,\alpha \right) $, there is $C$ depending on $%
d,\alpha ,b,G,p,\rho $ such that%
\begin{eqnarray*}
\mathbf{E}\left[ \sup_{t}\left\vert X_{t}^{n}-X_{t}\right\vert ^{p}\right]
&\leq &C\left( n/\ln n\right) ^{-p/\alpha }\text{ if }0<p<\alpha \in \left(
1,2\right) , \\
\mathbf{E}\left[ \sup_{t}\left\vert X_{t}^{n}-X_{t}\right\vert ^{p}\right]
&\leq &C\left[ n/\left( \ln n\right) ^{2}\right] ^{-p}\text{ if }0<p<\alpha
=1.
\end{eqnarray*}

(ii) If $\alpha =1$, and $\rho \left( y\right) =\rho \left( -y\right) ,y\in 
\mathbf{R}^{d},$ then there is $C$ depending on $d,\alpha ,b,G,p,\rho $ such
that%
\begin{equation*}
\mathbf{E}\left[ \sup_{t}\left\vert X_{t}^{n}-X_{t}\right\vert ^{p}\right]
\leq C\left( n/\ln n\right) ^{-p}\text{ if }0<p<\alpha =1.
\end{equation*}
\end{proposition}

We derive the following rate of convergence in all Lipshitz case for (\ref%
{m1'}).

\begin{proposition}
\label{pro4}Let $\alpha \in \lbrack 1,2),\rho $ be nonnegative bounded.
Assume $b$ and $G$ are bounded Lipshitz functions. Then

(i) For each $p\in \left( 0,\alpha \right) $, there is $C$ depending on $%
d,\alpha ,b,G,p,\rho $ such that%
\begin{equation*}
\mathbf{E}\left[ \sup_{t}\left\vert Y_{t}^{n}-Y_{t}\right\vert ^{p}\right]
\leq C\left\{ 
\begin{array}{ll}
\left( n/\ln n\right) ^{-p/\alpha } & \text{if }0<p<\alpha \in \left(
1,2\right) , \\ 
\left[ n/\left( \ln n\right) ^{2}\right] ^{-p} & \text{if }0<p<\alpha =1, \\ 
\left[ n/\left( \ln n\right) ^{2}\right] ^{-1} & \text{if }p=\alpha , \\ 
n^{-1} & \text{if }p>\alpha%
\end{array}%
\right.
\end{equation*}

(ii)\ If $\alpha =1$, and $\rho \left( y\right) =\rho \left( -y\right) ,y\in 
\mathbf{R}^{d},$ then there is $C$ depending on $d,\alpha ,b,G,p,\rho $ such
that%
\begin{equation*}
\mathbf{E}\left[ \sup_{t}\left\vert Y_{t}^{n}-Y_{t}\right\vert ^{p}\right]
\leq C\left( n/\ln n\right) ^{-p}\text{ if }0<p<\alpha =1.
\end{equation*}
\end{proposition}

The rates above are in agreement with the subtle results obtained in \cite{j}
for (\ref{m1}) in the case $d=1,b=0,G\in C^{3}$.

An obvious consequence of Proposition \ref{pro3} is

\begin{corollary}
\label{cl1}Let $\alpha \in \lbrack 1,2),\rho $ be nonnegative bounded.
Assume $b$ and $G$ are bounded Lipshitz functions. Then

(i) there is $C$ depending on $d,\alpha ,b,G,p,\rho $ such that for each $%
\varphi \in C^{\beta }\left( \mathbf{R}^{d}\right) ,t\in \left[ 0,1\right] ,$%
\begin{eqnarray*}
\left\vert \mathbf{E}\varphi \left( X_{t}\right) -\mathbf{E}\varphi \left(
X_{t}^{n}\right) \right\vert &\leq &C\left\vert \varphi \right\vert _{\beta
}\left( n/\ln n\right) ^{-\beta /\alpha }\text{ if }\alpha \in \left(
1,2\right) , \\
\left\vert \mathbf{E}\varphi \left( X_{t}\right) -\mathbf{E}\varphi \left(
X_{t}^{n}\right) \right\vert &\leq &C\left\vert \varphi \right\vert _{\beta }%
\left[ n/(\ln n)^{2}\right] ^{-\beta }\text{ if }\alpha =1.
\end{eqnarray*}

(ii)\ If $\alpha =1$, and $\rho \left( y\right) =\rho \left( -y\right) ,y\in 
\mathbf{R}^{d},$ then there is $C$ depending on $d,\alpha ,b,G,p,\rho $ such
that for each $\varphi \in C^{\beta }\left( \mathbf{R}^{d}\right) ,t\in %
\left[ 0,1\right] ,$%
\begin{equation*}
\left\vert \mathbf{E}\varphi \left( X_{t}\right) -\mathbf{E}\varphi \left(
X_{t}^{n}\right) \right\vert \leq C\left\vert \varphi \right\vert _{\beta
}\left( n/\ln n\right) ^{-\beta }.
\end{equation*}
\end{corollary}

Our note is organized as follows. In section 2, notation is introduced,
primary analytic tools are discussed and some auxiliary results are
presented. In section 3, we prove Propositions \ref{pro2}-\ref{pro4}.

\section{Notation and Auxiliary Results}

\subsection{Notation}

$\mathbf{R}_{0}^{d}:=\mathbf{R}^{d}\backslash \{0\}$. Denote $H_{T}=\left[
0,T\right] \times \mathbf{R}^{d},0\leq T\leq 1$. $I_{d}$ is the $d\times d$%
-identity matrix.

For any $x,y\in\mathbb{R}^d$, we write 
\begin{equation*}
\left(x,y\right)=\sum^d_{i=1}x_iy_i, \quad \left\vert
x\right\vert=\left(x,x\right)^{1/2}.
\end{equation*}

For a function $u=u\left( t,x\right) $ on $H$, we denote its partial
derivatives by $\partial _{t}u=\partial u/\partial t$, $\partial
_{i}u=\partial u/\partial x_{i}$, $\partial _{ij}^{2}u=\partial
^{2}u/\partial x_{i}x_{j}$, and denote its gradient with respect to $x$ by $%
\nabla u=\left( \partial _{1}u,\ldots ,\partial _{d}u\right) $ and $%
D^{|\gamma |}u=\partial ^{|\gamma |}u/\partial x_{1}^{\gamma _{1}}\ldots
\partial x_{d}^{\gamma _{d}}$, where $\gamma =\left( \gamma _{1},\ldots
,\gamma _{d}\right) \in \mathbf{N}^{d}$ is a multi-index. Meanwhile, we
write 
\begin{eqnarray*}
\left\vert u\right\vert _{0} &=&\sup_{t,x}\left\vert u\left( t,x\right)
\right\vert , \\
\left[ u\right] _{\beta } &=&\sup_{t,x,h\neq 0}\frac{\left\vert u\left(
t,x+h\right) -u\left( t,x\right) \right\vert }{\left\vert h\right\vert
^{\beta }}\quad \text{if}\quad \beta \in \left( 0,1\right) , \\
\left[ u\right] _{\beta } &=&\sup_{t,x,h\neq 0}\frac{\left\vert u\left(
t,x+h\right) -u\left( t,x\right) \right\vert }{\left\vert h\right\vert }%
\quad \text{if}\quad \beta =1.
\end{eqnarray*}

For $\beta =\left[ \beta \right] +\{\beta \}>0$, where $\left[ \beta \right]
\in \mathbf{N}$ is the greatest integer that is less than or equal to $\beta 
$ and $\{\beta \}\in \left( 0,1\right) $, $C^{\beta }\left( H_{T}\right) $
denotes the space of measurable functions $u$ on $H_{T}$ such that the norm 
\begin{equation*}
\left\vert u\right\vert _{\beta }=\sum_{\left\vert \gamma \right\vert \leq 
\left[ \beta \right] }\left\vert D^{\gamma }u\right\vert
_{0}+\sum_{\left\vert \gamma \right\vert =\left[ \beta \right] }\left[
D^{\gamma }u\right] _{\{\beta \}}<\infty .
\end{equation*}%
Analogous definitions apply to functions on $\mathbf{R}^{d}$, and $C^{\beta
}\left( \mathbf{R}^{d}\right) $ denotes the corresponding function space.

For a $d\times d$ matrix $G\left( x\right) $ on $\mathbf{R}^{d}$, we define
its norm to be the operator norm, i.e., 
\begin{equation*}
\left\vert G\left( x\right) \right\vert :=\sup_{y\in \mathbf{R}%
^{d},\left\vert y\right\vert =1}\left\vert G\left( x\right) y\right\vert ,
\end{equation*}%
and 
\begin{equation*}
\Vert G\Vert :=\sup_{x\in \mathbf{R}^{d}}\left\vert G\left( x\right)
\right\vert .
\end{equation*}%
In our note, $\Vert G\Vert $ is assumed to be finite and that implies each
entry $\left\vert G_{ij}\right\vert _{0}\leq \Vert G\Vert $.

Because Lipshitz continuity implies differentiability almost everywhere, we
write $\left\vert \nabla G\right\vert _{\infty }$ to denote the Lipshitz
constant of $G$, even if $G$ is not specified to be differentiable.

At last, $C=C\left( \cdot ,\ldots ,\cdot \right) $ denotes constants
depending only on quantities appearing in parentheses, but it may represent
different values in different contexts.

\subsection{Auxiliary Results}

\subsubsection{Backward Kolmogorov equations in H\"{o}lder classes}

We will rely on some results about backward Kolmogorov equations. For
convenience, we summarize assumptions that will be needed as follows:

\textbf{A}$(K,c_{0})$. (i) \textbf{S}$\left( c_{0}\right) $ holds and for
the same $c_{0}$, 
\begin{equation*}
\left\vert \det G\left( x\right) \right\vert \geq c_{0},x\in \mathbf{R}^{d};
\end{equation*}%
(ii) There is a constant $K$ such that 
\begin{equation*}
\left\vert \left\vert G\right\vert \right\vert +\left\vert \nabla
G\right\vert _{\infty }\leq K,\quad 0\leq \rho \left( y\right) \leq K,y\in 
\mathbf{R}^{d}.
\end{equation*}

Define for $v\in C_{0}^{\infty }\left( \mathbf{R}^{d}\right), x\in \mathbf{R}%
^{d}$, 
\begin{equation}
Lv\left( x\right) =\int_{\left\vert y\right\vert \leq 1}\left[ v\left(
x+G\left( x\right) y\right) -v\left( x\right) -\left( \nabla v\left(
x\right) \cdot G\left( x\right) y\right) \right] \rho \left( y\right) \frac{%
dy}{|y|^{d+\alpha }}.  \label{e0}
\end{equation}

\begin{proposition}
\label{p1}Let $\alpha \in \lbrack 1,2)$, $\mu \in (0,1)$, $\tilde{b}=\left( 
\tilde{b}^{k}\right) _{1\leq k\leq d}$ with $\tilde{b}^{k}\in C^{\mu }\left( 
\mathbf{R}^{d}\right) $, $\left\vert \tilde{b}^{k}\right\vert _{\mu }\leq K$ 
$\forall k$, and Assumption \textbf{A}$(K,c_{0})$\textbf{\ }hold. Let 
\begin{equation*}
\left\vert \rho \left( y\right) -\rho \left( z\right) \right\vert \leq
K\left\vert y-z\right\vert ^{\beta }\text{ for all }\left\vert y\right\vert
=\left\vert z\right\vert =1.
\end{equation*}

Then for any $f\in C^{\mu }\left( H_{1}\right) $, there exists a unique
solution $u\in C^{\alpha +\mu }\left( H_{1}\right) $ to the parabolic
equation 
\begin{eqnarray}
\qquad \partial _{t}u\left( t,x\right) &=&Lu\left( t,x\right) +\tilde{b}%
\left( x\right) \cdot \nabla u\left( t,x\right) +f\left( t,x\right) ,\quad
\left( t,x\right) \in H_{1},  \label{e1} \\
u\left( 0,x\right) &=&0,\quad x\in \mathbf{R}^{d}.  \notag
\end{eqnarray}

Moreover, there is a constant $C=C\left( \alpha ,\mu ,d,K,c_{0}\right) $
such that 
\begin{equation*}
\left\vert u\right\vert _{\alpha +\mu }\leq C\left\vert f\right\vert _{\mu },
\end{equation*}%
and for all $s\leq t\leq 1$, 
\begin{equation*}
\left\vert u\left( t,\cdot \right) -u\left( s,\cdot \right) \right\vert _{%
\frac{\alpha }{2}+\mu }\leq C\left( t-s\right) ^{1/2}\left\vert f\right\vert
_{\mu }.
\end{equation*}
\end{proposition}

\begin{proof}
We apply \text{Theorem} 4 in \cite{mp1} with $\mathcal{L}=A+B,$ where 
\begin{eqnarray*}
&&Au\left( t,x\right) =\int [u\left( t,x+G\left( x\right) y\right) -u\left(
t,x\right) \\
&&\qquad \qquad \qquad -\left( \nabla u\left( t,x\right) \cdot G\left(
x\right) y\right) \chi _{\alpha }\left( y\right) ]\rho \left( y\right) \frac{%
dy}{|y|^{d+\alpha }}, \\
&&Bu\left( t,x\right) =\bar{b}\left( x\right) \cdot \nabla u\left( t,x\right)
\\
&&\qquad \qquad \qquad -\int_{\left\vert y\right\vert >1}\left[ u\left(
t,x+G\left( x\right) y\right) -u\left( t,x\right) \right] \rho \left(
y\right) \frac{dy}{|y|^{d+\alpha }},\left( t,x\right) \in H_{1},
\end{eqnarray*}
$\chi _{\alpha }\left( y\right) =1$ if $\alpha \in \left( 1,2\right) ,\chi
_{\alpha }\left( y\right) =\chi _{\left\{ \left\vert y\right\vert \leq
1\right\} }\left( y\right) $ if $\alpha =1$, and%
\begin{equation*}
\bar{b}\left( x\right) =\tilde{b}\left( x\right) +1_{\alpha \in \left(
1,2\right) }G\left( x\right) \int_{\left\vert y\right\vert >1}y\rho \left(
y\right) \frac{dy}{\left\vert y\right\vert ^{d+\alpha }},x\in \mathbf{R}^{d}.
\end{equation*}
Using the symmetry assumption on $\rho $ and changing variables of
integration, we see that%
\begin{equation*}
Au\left( t,x\right) =\int \left[ u\left( t,x+y\right) -u\left( t,x\right)
-\left( \nabla u\left( t,x\right) \cdot y\right) \chi _{\alpha }\left(
y\right) \right] m\left( x,y\right) \frac{dy}{|y|^{d+\alpha }},
\end{equation*}%
where for $x\in \mathbf{R}^{d},y\in \mathbf{R}_{0}^{d},$ 
\begin{equation*}
m\left( x,y\right) =\frac{\rho \left( G^{-1}\left( x\right) y\right) }{%
\left\vert \det G\left( x\right) \right\vert \left\vert G^{-1}\left(
x\right) \frac{y}{\left\vert y\right\vert }\right\vert ^{d+\alpha }}:=%
\widetilde{m}\left( x,y\right) \rho \left( G^{-1}\left( x\right) y\right) .
\end{equation*}%
First we verify assumptions of Theorem 4 in \cite{mp1} for $m\left(
x,y\right) $. Obviously, 
\begin{equation*}
\left\vert G\left( x\right) y\right\vert \leq K\left\vert y\right\vert ,%
\text{ }x,y\in \mathbf{R}^{d},
\end{equation*}%
which implies $\left\vert y\right\vert \leq K\left\vert G\left( x\right)
^{-1}y\right\vert $ and thus $\left\vert G^{-1}\left( x\right) \frac{y}{%
\left\vert y\right\vert }\right\vert \geq 1/K,x\in \mathbf{R}^{d},y\in 
\mathbf{R}_{0}^{d}$. Therefore, 
\begin{equation*}
\left\vert m\left( x,y\right) \right\vert \leq \frac{K^{d+\alpha +1}}{c_{0}},%
\text{ }x\in \mathbf{R}^{d},y\in \mathbf{R}_{0}^{d}.
\end{equation*}%
On the other hand, it's obvious that $\det G\left( x\right) $ is bounded and
Lipshitz with $c_{0}\leq \left\vert \det G\left( x\right) \right\vert \leq
K^{d}d!$, which implies both $\frac{1}{\left\vert \det G\left( x\right)
\right\vert }$ and $\left\vert G^{-1}\left( x\right) \frac{y}{\left\vert
y\right\vert }\right\vert =\left\vert \frac{adj\left( G\left( x\right)
\right) }{\det G\left( x\right) }\frac{y}{\left\vert y\right\vert }%
\right\vert $ are Lipshitz in $x$ uniformly over $y$. With 
\begin{equation}
K^{-1}\leq \left\vert G^{-1}\left( x\right) \frac{y}{\left\vert y\right\vert 
}\right\vert \leq \frac{K^{d-1}\left( d-1\right) !{d}^{3/2}}{c_{0}}%
=:c_{1},x\in \mathbf{R},y\in \mathbf{R}_{0}^{d},  \label{e2}
\end{equation}%
we can conclude $\widetilde{m}\left( x,y\right) $ is Lipshitz uniformly over 
$y$. Meanwhile, recall that $\rho $ is $\mu $-H\"{o}lder continuous and
0-homogeneous. Hence 
\begin{eqnarray*}
&&\frac{\left\vert \rho \left( G^{-1}\left( x+h\right) y\right) -\rho \left(
G^{-1}\left( x\right) y\right) \right\vert }{\left\vert h\right\vert ^{\mu }}
\\
&=&\frac{\left\vert \rho \left( G^{-1}\left( x+h\right) \frac{y}{\left\vert
y\right\vert }\right) -\rho \left( G^{-1}\left( x\right) \frac{y}{\left\vert
y\right\vert }\right) \right\vert }{\left\vert G^{-1}\left( x+h\right) \frac{%
y}{\left\vert y\right\vert }-G^{-1}\left( x\right) \frac{y}{\left\vert
y\right\vert }\right\vert ^{\mu }}\cdot \frac{\left\vert G^{-1}\left(
x+h\right) \frac{y}{\left\vert y\right\vert }-G^{-1}\left( x\right) \frac{y}{%
\left\vert y\right\vert }\right\vert ^{\mu }}{\left\vert h\right\vert ^{\mu }%
} \\
&\leq &K\left\vert \nabla \left( G^{-1}\right) \right\vert _{\infty }^{\mu },
\end{eqnarray*}%
and therefore $m\left( x,y\right) $ is $\mu $-continuous in $x$ uniformly
over $y$.

When $\alpha =1$, according to \eqref{sym}, 
\begin{eqnarray*}
\int_{r<\left\vert y\right\vert \leq 1}ym\left( x,y\right) \frac{dy}{%
\left\vert y\right\vert ^{d+\alpha }} &=&\int_{r<\left\vert y\right\vert
\leq 1}\frac{y\rho \left( G^{-1}\left( x\right) y\right) }{\left\vert \det
G\left( x\right) \right\vert \left\vert G^{-1}\left( x\right) \frac{y}{%
\left\vert y\right\vert }\right\vert ^{d+\alpha }}\frac{dy}{\left\vert
y\right\vert ^{d+\alpha }} \\
&=&\int_{r<\left\vert y\right\vert \leq 1}\frac{-y\rho \left( G^{-1}\left(
x\right) y\right) }{\left\vert \det G\left( x\right) \right\vert \left\vert
G^{-1}\left( x\right) \frac{y}{\left\vert y\right\vert }\right\vert
^{d+\alpha }}\frac{dy}{\left\vert y\right\vert ^{d+\alpha }}=0.
\end{eqnarray*}%
Note that, there is $c_{2}=c_{2}\left( c_{0},\alpha,K,d\right) $ such that $%
m\left( x,y\right) \geq c_{2},~\forall x\in \mathbf{R}^{d},\forall y\in 
\mathbf{R}_{0}^{d}$. Then, \textbf{Assumption A} in Theorem 4 of \cite{mp1}
is satisfied.

Let $U=\left\{ y:\left\vert y\right\vert >1\right\} ,U_{1}=\left\{
y:\left\vert y\right\vert \leq 1\right\} $, and $c\left( x,y\right) =G\left(
x\right) y$ if $\left\vert y\right\vert >1$, $c\left( x,y\right) =0$
otherwise. Then $Bu\left( t,x\right) $ can be written as 
\begin{eqnarray*}
&&Bu\left( t,x\right) =\bar{b}\left( x\right) \cdot \nabla u\left(
t,x\right) -\int_{U}[u\left( t,x+c\left( x,y\right) \right) -u\left(
t,x\right) \\
&&\qquad \qquad -\left( \nabla u\left( t,x\right) \cdot c\left( x,y\right)
\right) 1_{U_{1}}\left( y\right) ]\rho \left( y\right) \frac{dy}{%
|y|^{d+\alpha }}.
\end{eqnarray*}%
By (\ref{e2}), $\left\vert y\right\vert \leq c_{1}\left\vert G\left(
x\right) y\right\vert $ for all $x,y\in \mathbf{R}^{d}$, thus $\left\vert
c\left( x,y\right) \right\vert \geq c_{1}^{-1}$ for all $x,y\in \mathbf{R}%
^{d}$. Then by choosing $\varepsilon <c_{1}^{-1}$, we have 
\begin{equation*}
\int_{\left\vert c\left( x,y\right) \right\vert \leq \varepsilon }\left\vert
c\left( x,y\right) \right\vert ^{\alpha }\rho \left( y\right) \frac{dy}{%
\left\vert y\right\vert ^{d+\alpha }}=0,\quad \forall x\in \mathbf{R}^{d}.
\end{equation*}%
Hence, \textbf{Assumption} \textbf{B1 }of Theorem 4 in \cite{mp1} holds.

We might as well set $K>1$. Now, for $\left\vert h\right\vert \leq 1,$ 
\begin{eqnarray*}
&&\int_{\left\vert y\right\vert >1}\left[ \left\vert c\left( x,y\right)
-c\left( x+h,y\right) \right\vert \wedge 1\right] \rho \left( y\right) \frac{%
dy}{\left\vert y\right\vert ^{d+\alpha }} \\
&\leq &K^{2}\int_{\left\vert y\right\vert >1}\left[ \left\vert h\right\vert
\left\vert y\right\vert \wedge 1\right] \frac{dy}{\left\vert y\right\vert
^{d+\alpha }}=K^{2}\int_{\left\vert h\right\vert \left\vert y\right\vert
>\left\vert h\right\vert }\left[ \left\vert h\right\vert \left\vert
y\right\vert \wedge 1\right] \frac{dy}{\left\vert y\right\vert ^{d+\alpha }}
\\
&=&K^{2}\left\vert h\right\vert ^{\alpha }\int_{\left\vert z\right\vert
>\left\vert h\right\vert }\left[ \left\vert z\right\vert \wedge 1\right] 
\frac{dz}{\left\vert z\right\vert ^{d+\alpha }}\leq C\left\vert h\right\vert
\left( 1+1_{\alpha =1}\left\vert \ln \left\vert h\right\vert \right\vert
\right)
\end{eqnarray*}%
for some $C=C\left( \alpha,K,d\right)$, Therefore\textbf{\ Assumption B2} of
Theorem 4 in \cite{mp1} is satisfied and our statement holds.
\end{proof}

Now, consider the backward Kolmogorov equation 
\begin{eqnarray}
\partial _{t}v\left( t,x\right) +\tilde{b}\left( x\right) \cdot \nabla
v\left( t,x\right) +Lv\left( t,x\right) &=&f\left( x\right) ,\quad (t,x)\in
H_{T},  \label{m3} \\
v\left( T,x\right) &=&0,\quad x\in \mathbf{R}^{d},  \notag
\end{eqnarray}%
where $L$ is defined as (\ref{e0}). If $u$ solves (\ref{e1}) in $H_{1}$ with 
$f=f\left( x\right) ,x\in \mathbf{R}^{d}$, then $v\left( t,x\right) =u\left(
T-t,x\right) ,T-1\leq t\leq T,x\in \mathbf{R}^{d}$, solves (\ref{m3}) with $%
T\in \left[ 0,1\right] $. The following statement is an obvious consequence
of Proposition \ref{p1}.

\begin{corollary}
\label{lem5}Let $\alpha \in \lbrack 1,2)$, $\mu \in (0,1)$, $\tilde{b}%
=\left( \tilde{b}^{k}\right) _{1\leq k\leq d}$ with $\tilde{b}^{k}\in C^{\mu
}\left( \mathbf{R}^{d}\right) $, $\left\vert \tilde{b}^{k}\right\vert _{\mu
}\leq K$ $\forall k$, and Assumption \textbf{A}$(K,c_{0})$\textbf{\ }hold.
Let 
\begin{equation*}
\left\vert \rho \left( y\right) -\rho \left( z\right) \right\vert \leq
K\left\vert y-z\right\vert ^{\beta }\text{ for all }\left\vert y\right\vert
=\left\vert z\right\vert =1.
\end{equation*}%
Then for any $f\in C^{\mu }\left( \mathbf{R}^{d}\right) $ and $T\in \left[
0,1\right] $, there exists a unique solution $v\in C^{\alpha +\mu }\left(
H_{T}\right) $ to (\ref{m3}). Moreover, there is a constant $C=C\left(
\alpha ,\mu ,d,K,c_{0}\right) $, independent of $T,$ such that 
\begin{equation*}
\left\vert v\right\vert _{\alpha +\mu }\leq C\left\vert f\right\vert _{\mu },
\end{equation*}%
and for all $0\leq s\leq t\leq T$, 
\begin{equation*}
\left\vert v\left( t,\cdot \right) -v\left( s,\cdot \right) \right\vert _{%
\frac{\alpha }{2}+\mu }\leq C\left( t-s\right) ^{1/2}\left\vert f\right\vert
_{\mu }.
\end{equation*}
\end{corollary}

\subsubsection{Some estimates of stochastic integrals and driving processes}

We present here some stochastic integral estimates related to stable type
point measures. Let $\mathcal{P=P}\left( \mathbb{F}\right) $ be predictable $%
\sigma $-algebra on $[0,1)\times \Omega .$

Let $F:[0,1)\times \Omega \times \mathbf{R}_{0}^{d}\rightarrow \mathbf{R}%
^{m} $ be a $\mathcal{P\times B}\left( \mathbf{R}_{0}^{d}\right) $%
-measurable vector function, 
\begin{equation*}
F=F_{r}\left( y\right) =\left( F_{r}^{i}\left( y\right) \right) _{1\leq
i\leq m},r\in \lbrack 0,1),y\in \mathbf{R}_{0}^{d},
\end{equation*}%
such that for any $T\in \lbrack 0,1)$ a.s.,%
\begin{equation}
\int_{0}^{T}\int_{\left\vert y\right\vert \leq 1}\left\vert F_{r}\left(
y\right) \right\vert ^{2}\rho \left( y\right) \frac{dydr}{\left\vert
y\right\vert ^{d+\alpha }}<\infty .  \label{fo6}
\end{equation}%
Let $0\leq S\leq T\leq 1.$ Consider the stochastic process 
\begin{equation*}
U_{t}=\int_{S}^{t}\int_{\left\vert y\right\vert \leq 1}F_{r}\left( y\right)
q\left( dr,dy\right) ,t\in \left[ S,T\right] .
\end{equation*}%
Note $U_{t}$ is well defined because of (\ref{fo6}).

The following estimates hold.

\begin{lemma}
\label{ele1}Let $\alpha \in \lbrack 1,2),p\in \left( \alpha ,\infty \right)
\,,0\leq \rho \left( y\right) \leq K,y\in \mathbf{R}^{d}$. Assume there is a
predictable nonnegative process $\bar{F}_{r},r\in \lbrack S,T],$ such that%
\begin{equation*}
\left\vert F_{r}\left( y\right) \right\vert \leq \bar{F}_{r}\left\vert
y\right\vert ,r\in \lbrack S,T],y\in \mathbf{R}^{d}.
\end{equation*}
Then there is $C=C\left( d,p,\alpha ,K\right) $ such that%
\begin{equation*}
\mathbf{E}\left[ \sup_{S\leq t\leq T}\left\vert U_{t}\right\vert ^{p}\right]
\leq C\mathbf{E}\int_{S}^{T}\left\vert \bar{F}_{r}\right\vert ^{p}dr .
\end{equation*}
\end{lemma}

\begin{proof}
If $p\geq 2$, then by Lemma \ref{ele5}(i) (e.g. Lemma 4.1 in \cite{lm}),%
\begin{eqnarray}
&&\mathbf{E}\left[ \sup_{S\leq t\leq T}\left\vert U_{t}\right\vert ^{p}%
\right]  \label{fo3} \\
&\leq &C\mathbf{E}\left[ \left( \int_{S}^{T}\int_{\left\vert y\right\vert
\leq 1}\left\vert \bar{F}_{r}y\right\vert ^{2}\frac{dydr}{\left\vert
y\right\vert ^{d+\alpha }}\right) ^{p/2}+\int_{S}^{T}\int_{\left\vert
y\right\vert \leq 1}\left\vert \bar{F}_{r}y\right\vert ^{p}\frac{dydr}{%
\left\vert y\right\vert ^{d+\alpha }}\right]  \notag \\
&\leq &C\mathbf{E}\int_{S}^{T}\left\vert \bar{F}_{r}\right\vert ^{p}dr. 
\notag
\end{eqnarray}

If $p\in \left( \alpha ,2\right) $, then by Burkholder-Davis-Gundy (BDG)
inequality, see Remark \ref{re1}, 
\begin{eqnarray}
&&\mathbf{E}\left[ \sup_{S\leq t\leq T}\left\vert U_{t}\right\vert ^{p}%
\right] \leq C\mathbf{E}\left[ \left( \int_{S}^{T}\int_{\left\vert
y\right\vert \leq 1}\left\vert \bar{F}_{r}y\right\vert ^{2}N\left(
dr,dy\right) \right) ^{p/2}\right]  \label{fo4} \\
&\leq &C\mathbf{E}\left[ \int_{S}^{T}\int_{\left\vert y\right\vert \leq
1}\left\vert \bar{F}_{r}y\right\vert ^{p}N\left( dr,dy\right) \right] \leq C%
\mathbf{E}\int_{S}^{T}\left\vert \bar{F}_{r}\right\vert ^{p}dr.  \notag
\end{eqnarray}
\end{proof}

\begin{lemma}
\label{ele11}Let $0\leq \rho \left( y\right) \leq K,y\in \mathbf{R}^{d}$.
Assume there is a predictable nonnegative process $\bar{F}_{r},r\in \lbrack
S,T],$ such that%
\begin{equation*}
\left\vert F_{r}\left( y\right) \right\vert \leq \bar{F}_{r}\left\vert
y\right\vert ,r\in \lbrack S,T],y\in \mathbf{R}^{d}.
\end{equation*}
(i) Let $\alpha \in (1,2),p\in (0,\alpha )$. Then there is $C=C\left(
d,p,\alpha ,K\right) $ such that%
\begin{equation*}
\mathbf{E}\left[ \sup_{S\leq t\leq T}\left\vert U_{t}\right\vert ^{p}\right]
\leq C\left( \mathbf{E}\left[ \int_{S}^{T}\left\vert \bar{F}_{r}\right\vert
^{\alpha }dr\right] \right) ^{p/\alpha }.
\end{equation*}
(ii) Let $\alpha \in \lbrack 1,2)$, $\bar{F}_{r}\leq M$ a.s. for some
constant $M>0$ and $\mathbf{E}\int_{S}^{T}\bar{F}_{r}^{\alpha }dr<1$. Then
there is $C=C\left( d,\alpha ,M,K\right) $ such that%
\begin{equation*}
\mathbf{E}\left[ \sup_{S\leq t\leq T}\left\vert U_{t}\right\vert ^{\alpha }%
\right] \leq C\mathbf{E}\int_{S}^{T}\bar{F}_{r}^{\alpha }dr\left[
1+\left\vert \ln \left( \mathbf{E}\int_{S}^{T}\bar{F}_{r}^{\alpha }dr\right)
\right\vert \right] .
\end{equation*}
\end{lemma}

\begin{proof}
For any $\varepsilon >0,$%
\begin{eqnarray*}
U_{t} &=&\int_{S}^{t}\int_{\bar{F}_{r}\left\vert y\right\vert \leq
\varepsilon ,\left\vert y\right\vert \leq 1}\cdots +\int_{S}^{t}\int_{\bar{F}%
_{r}\left\vert y\right\vert >\varepsilon ,\left\vert y\right\vert \leq
1}\cdots \\
&:=&U_{t}^{1}+U_{t}^{2},t\in \left[ S,T\right] .
\end{eqnarray*}%
Let $0<p\leq \alpha \in \lbrack 1,2)$. By Remark \ref{re1} (Corollary II in 
\cite{le}), 
\begin{eqnarray}
\mathbf{\qquad }\mathbf{E}\left[ \sup_{S\leq t\leq T}\left\vert
U_{t}^{1}\right\vert ^{p}\right] &\leq &C\mathbf{E}\left[ \left(
\int_{S}^{T}\int_{\left\vert \bar{F}_{r}y\right\vert \leq \varepsilon
,\left\vert y\right\vert \leq 1}\left\vert F_{r}\left( y\right) \right\vert
^{2}\frac{dydr}{\left\vert y\right\vert ^{d+\alpha }}\right) ^{p/2}\right] 
\notag \\
&\leq &C\mathbf{E}\left[ \left( \int_{S}^{T}\int_{\left\vert \bar{F}%
_{r}y\right\vert \leq \varepsilon }\left\vert \bar{F}_{r}y\right\vert ^{2}%
\frac{dydr}{\left\vert y\right\vert ^{d+\alpha }}\right) ^{p/2}\right]
\label{fo0} \\
&\leq &C\varepsilon ^{\left( 1-\alpha /2\right) p}\left( \mathbf{E}%
\int_{S}^{T}\bar{F}_{r}^{\alpha }dr\right) ^{p/2}.  \notag
\end{eqnarray}%
Let $p\in \lbrack 1,2)$. Then by BDG inequality, Remark \ref{re1}, 
\begin{eqnarray}
\mathbf{\qquad }\mathbf{E}\left[ \sup_{S\leq t\leq T}\left\vert
U_{t}^{2}\right\vert ^{p}\right] &\leq &C\mathbf{E}\left[ \left(
\int_{S}^{T}\int_{\left\vert \bar{F}_{r}y\right\vert >\varepsilon
,\left\vert y\right\vert \leq 1}\left\vert F_{r}\left( y\right) \right\vert
^{2}N\left( dr,dy\right) \right) ^{p/2}\right]  \notag \\
&\leq &C\mathbf{E}\left[ \int_{S}^{T}\int_{\left\vert \bar{F}%
_{r}y\right\vert >\varepsilon ,\left\vert y\right\vert \leq 1}\left\vert 
\bar{F}_{r}y\right\vert ^{p}\frac{dydr}{\left\vert y\right\vert ^{d+\alpha }}%
\right] .  \label{fo1}
\end{eqnarray}

If $p\in \lbrack 1,\alpha ),\alpha \in \left( 1,2\right) $, then 
\begin{eqnarray*}
\mathbf{E}\left[ \sup_{S\leq t\leq T}\left\vert U_{t}^{2}\right\vert ^{p}%
\right]\leq C\varepsilon ^{-\left( \alpha -p\right) }\mathbf{E}\int_{S}^{T}%
\bar{F}_{r}^{\alpha }dr.
\end{eqnarray*}%
Taking $\varepsilon =\left( \mathbf{E}\int_{S}^{T}\bar{F}_{r}^{\alpha
}dr\right) ^{1/\alpha }$ and combining with (\ref{fo0}) , 
\begin{equation}
\mathbf{E}\left[ \sup_{S\leq t\leq T}\left\vert U_{t}\right\vert ^{p}\right]
\leq C\left( \mathbf{E}\int_{S}^{T}\bar{F}_{r}^{\alpha }dr\right) ^{p/\alpha
}.  \label{fo2}
\end{equation}

If $p\in \left( 0,1\right) ,\alpha \in \left( 1,2\right) $, then by H\"{o}%
lder inequality and (\ref{fo2}),%
\begin{equation*}
\mathbf{E}\left[ \sup_{S\leq t\leq T}\left\vert U_{t}\right\vert ^{p}\right]
\leq \left( \mathbf{E}\sup_{S\leq t\leq T}\left\vert U_{t}\right\vert
\right) ^{p}\leq C\left( \mathbf{E}\int_{S}^{T}\bar{F}_{r}^{\alpha
}dr\right) ^{p/\alpha }.
\end{equation*}

If $p=\alpha \in \lbrack 1,2)$, then, according to (\ref{fo1}),%
\begin{eqnarray*}
\mathbf{E}\left[ \sup_{S\leq t\leq T}\left\vert U_{t}^{2}\right\vert ^{p}%
\right] &\leq &C\mathbf{E}\left[ \int_{S}^{T}\int_{\varepsilon< \left\vert 
\bar{F}_{r}^{n}y\right\vert \leq M }\left\vert \bar{F}_{r}y\right\vert
^{\alpha }\frac{dydr}{\left\vert y\right\vert ^{d+\alpha }}\right] \\
&\leq &C\left( 1+|\ln \varepsilon |\right)\mathbf{E}\int_{S}^{T}\bar{F}%
_{r}^{\alpha }dr .
\end{eqnarray*}%
Taking $\varepsilon =\mathbf{E}\int_{S}^{T}\bar{F}_{r}^{\alpha }dr$ and
combining with (\ref{fo0}), we see that 
\begin{equation*}
\mathbf{E}\left[ \sup_{S\leq t\leq T}\left\vert U_{t}\right\vert ^{\alpha }%
\right] \leq C\mathbf{E}\int_{S}^{T}\bar{F}_{r}^{\alpha }dr\left[ 1+|\ln
\left( \mathbf{E}\int_{S}^{T}\bar{F}_{r}^{\alpha }dr\right) |\right] .
\end{equation*}
\end{proof}

\begin{lemma}
\label{ele12}Let $0\leq \rho \left( y\right) \leq K,y\in \mathbf{R}^{d}$.
Assume there is a predictable nonnegative process $\bar{F}_{r},r\in \lbrack
S,T],$ such that%
\begin{equation*}
\left\vert F_{r}\left( y\right) \right\vert \leq \bar{F}_{r}\left\vert
y\right\vert ,r\in \lbrack S,T],y\in \mathbf{R}^{d}.
\end{equation*}%
(i) Let $\alpha =1,p\in \left( 0,1\right) $, and $\bar{F}_{r}\leq M$ a.s.
for some constant $M>0.$ Then there is $C=C\left( d,p,\alpha ,M,K\right) $
such that%
\begin{equation*}
\mathbf{E}\left[ \sup_{S\leq t\leq T}\left\vert U_{t}\right\vert ^{p}\right]
\leq C\left( \mathbf{E}\int_{S}^{T}\bar{F}_{r}^{\alpha }dr\right) ^{p}\left[
1+\left\vert \ln \left( \mathbf{E}\int_{S}^{T}\bar{F}_{r}^{\alpha }dr\right)
\right\vert \right] ^{p}.
\end{equation*}
(ii) Let $\alpha =1,p\in \left( 0,1\right) $. Assume $\rho \left( y\right)
=\rho \left( -y\right) ,y\in \mathbf{R}^{d}.$ Suppose there exists a
predictable $m\times d$ matrix valued function $H_{r},r\in \left[ S,T\right]
,$ such that a.s.%
\begin{equation*}
\left\vert F_{r}\left( y\right) -H_{r}y\right\vert \leq M\bar{F}%
_{r}\left\vert y\right\vert ^{1+\beta ^{\prime }}\text{, }r\in \left[ S,T%
\right] ,\left\vert y\right\vert \leq 1,
\end{equation*}%
for some constants $M>0,\beta ^{\prime }>0.$ Then there is $C=C\left(
d,p,\alpha ,M,K\right) $ such that 
\begin{equation*}
\mathbf{E}\left[ \sup_{S\leq t\leq T}\left\vert U_{t}\right\vert ^{p}\right]
\leq C\left( \mathbf{E}\int_{S}^{T}\left\vert \bar{F}_{r}\right\vert
dr\right) ^{p}.
\end{equation*}
\end{lemma}

\begin{proof}
(i) Let $\alpha =1,p\in \left( 0,1\right) $. By H\"{o}lder inequality,%
\begin{equation*}
\mathbf{E}\left[ \sup_{S\leq t\leq T}\left\vert U_{t}\right\vert ^{p}\right]
\leq \left( \mathbf{E}\sup_{S\leq t\leq T}\left\vert U_{t}\right\vert
\right) ^{p},
\end{equation*}%
and the estimate follows by Lemma \ref{ele11}(ii).

(ii) For $\varepsilon >0,$ we decompose%
\begin{eqnarray*}
U_{t} &=&\int_{S}^{t}\int_{\bar{F}_{r}\left\vert y\right\vert \leq
\varepsilon ,\left\vert y\right\vert \leq 1}\cdots+\int_{S}^{t}\int_{\bar{F}%
_{r}\left\vert y\right\vert >\varepsilon ,\left\vert y\right\vert \leq
1}\cdots \\
&:=&U_{t}^{1}+U_{t}^{2},t\in \left[ S,T\right] .
\end{eqnarray*}

Let $0<p<1$. By Remark \ref{re1} (Corollary II in \cite{le}), there is $%
C=C\left( K,d,p\right) $ such that 
\begin{eqnarray}
\mathbf{E}\left[ \sup_{S\leq t\leq T}\left\vert U_{t}^{1}\right\vert ^{p}%
\right] &\leq &C\mathbf{E}\left[ \left( \int_{S}^{T}\int_{\left\vert \bar{F}%
_{r}^{n}y\right\vert \leq \varepsilon }\left\vert \bar{F}_{r}y\right\vert
^{2}\frac{dydr}{\left\vert y\right\vert ^{d+1}}\right) ^{p/2}\right]
\label{fo01} \\
&\leq &C\varepsilon ^{p/2}\left( \mathbf{E}\int_{S}^{T}\bar{F}_{r}dr\right)
^{p/2}.  \notag
\end{eqnarray}

We decompose further 
\begin{eqnarray*}
U_{t}^{2} &=&\int_{S}^{t}\int_{\bar{F}_{r}\left\vert y\right\vert
>\varepsilon ,\left\vert y\right\vert \leq 1}F_{r}(y)N\left( dr,dy\right) \\
&&+\int_{S}^{t}\int_{\bar{F}_{r}\left\vert y\right\vert >\varepsilon
,\left\vert y\right\vert \leq 1}\left( H_{r}y-F_{r}(y)\right) \rho \left(
y\right) \frac{dydr}{\left\vert y\right\vert ^{d+1}} \\
&:=&U_{t}^{21}+U_{t}^{22},t\in \left[ S,T\right] .
\end{eqnarray*}%
Now,%
\begin{eqnarray*}
\mathbf{E}\left[ \sup_{S\leq t\leq T}\left\vert U_{t}^{21}\right\vert ^{p} %
\right] &\leq &\mathbf{E}\int_{S}^{T}\int_{\bar{F}_{r}\left\vert
y\right\vert >\varepsilon ,\left\vert y\right\vert \leq 1}\left\vert \bar{F}
_{r}y\right\vert ^{p}N\left( dr,dy\right) \\
&\leq &C\mathbf{E}\int_{S}^{T}\int_{\bar{F}_{r}\left\vert y\right\vert
>\varepsilon }\left\vert \bar{F}_{r}y\right\vert ^{p}\frac{dr}{\left\vert
y\right\vert ^{d+1}}\leq C\varepsilon ^{-\left( 1-p\right) }\mathbf{E}
\int_{S}^{T}\bar{F}_{r}dr,
\end{eqnarray*}
and 
\begin{eqnarray*}
\mathbf{E}\left[ \sup_{S\leq t\leq T}\left\vert U_{t}^{22}\right\vert ^{p}%
\right] &\leq &C\mathbf{E}\left( \int_{S}^{T}\int_{\bar{F}_{r}\left\vert
y\right\vert >\varepsilon ,\left\vert y\right\vert \leq 1}\left\vert
F_{r}(y)-H_{r}y\right\vert \frac{dy}{\left\vert y\right\vert ^{d+1}}%
dr\right) ^{p} \\
&\leq &C\mathbf{E}\left( \int_{S}^{T}\int_{\left\vert y\right\vert \leq 1}%
\bar{F}_{r}\left\vert y\right\vert ^{1+\beta ^{\prime }}\frac{dy}{\left\vert
y\right\vert ^{d+1}}dr\right) ^{p} \\
&\leq &C\mathbf{E}\left[ \left( \int_{S}^{T}\bar{F}_{r}dr\right) ^{p}\right]
\leq C\left( \mathbf{E}\int_{S}^{T}\bar{F}_{r}dr\right) ^{p}.
\end{eqnarray*}

Combining these estimates with (\ref{fo01}) and taking $\varepsilon =\mathbf{%
E}\int_{S}^{T}\bar{F}_{r}dr$, we see that for $\alpha =1,p\in \left(
0,1\right) ,$ there is $C=C\left( \alpha ,d,p,K,M\right) $ such that%
\begin{equation*}
\mathbf{E}\left[ \sup_{S\leq t\leq T}\left\vert U_{t}\right\vert ^{p}\right]
\leq C\left( \mathbf{E}\int_{S}^{T}\bar{F}_{r}dr\right) ^{p}.
\end{equation*}
\end{proof}

Again, let $F:[0,1)\times \Omega \times \mathbf{R}_{0}^{d}\rightarrow 
\mathbf{R}^{m}$ be a $\mathcal{P\times B}\left( \mathbf{R}^{d}\right)$%
-measurable vector function, 
\begin{equation*}
F=F_{r}\left( y\right) =\left( F_{r}^{i}\left( y\right) \right) _{1\leq
i\leq m},r\in \lbrack 0,1),y\in \mathbf{R}_{0}^{d},
\end{equation*}%
such that for any $T\in \lbrack 0,1)$ a.s.%
\begin{equation}
\int_{0}^{T}\int_{\left\vert y\right\vert >1}\left\vert F_{r}\left( y\right)
\right\vert \rho\left( y\right)\frac{dydr}{\left\vert y\right\vert
^{d+\alpha }}<\infty \text{ if }\alpha \in \left[ 1,2\right) .  \label{fo7}
\end{equation}
Let $0\leq S\leq T\leq 1.$ Consider the stochastic process 
\begin{equation*}
Z_{t}=\int_{S}^{t}\int_{\left\vert y\right\vert >1}F_{r}\left( y\right)
N\left( dr,dy\right) ,t\in \left[ S,T\right].
\end{equation*}%
Note $Z_t$ is well defined because of (\ref{fo7}).

Later we will need the following estimates as well.

\begin{lemma}
\label{ele2}Let $\alpha \in \lbrack 1,2),p\in \left( 0,\alpha \right) ,0\leq
\rho \left( y\right) \leq K,y\in \mathbf{R}^{d},0\leq S\leq T\leq 1$. Assume
there is a predictable nonnegative process $\bar{F}_{r},r\in \lbrack S,T],$
such that%
\begin{equation*}
\left\vert F_{r}\left( y\right) \right\vert \leq \bar{F}_{r}\left\vert
y\right\vert ,r\in \lbrack S,T],y\in \mathbf{R}^{d}.
\end{equation*}
Then there is $C=C\left( d,p,\alpha ,K\right) $ such that%
\begin{equation*}
\mathbf{E}\left[ \sup_{S\leq t\leq T}\left\vert Z_{t}\right\vert ^{p}\right]
\leq C\mathbf{E}\int_{S}^{T}\left\vert \bar{F}_{r}\right\vert ^{p}dr.
\end{equation*}
\end{lemma}

\begin{proof}
Let $p\in \left( 0,1\right) $. Then, according to Remark \ref{re3}, 
\begin{eqnarray*}
\mathbf{E}\left[ \sup_{S\leq t\leq T}\left\vert Z_{t}\right\vert ^{p}\right]
&\leq &C\mathbf{E}\left[ \int_{S}^{T}\int_{\left\vert y\right\vert
>1}|F_{r}\left( y\right) |^{p}\rho \left( y\right) \frac{dydr}{\left\vert
y\right\vert ^{d+\alpha }}\right] \\
&\leq &C\mathbf{E}\int_{S}^{T}\bar{F}_{r}^{p}dr.
\end{eqnarray*}

Let $p\in \lbrack 1,\alpha ),\alpha \in \left( 1,2\right) $. By Lemma \ref%
{ele5}(ii),%
\begin{eqnarray*}
\mathbf{E}\left[ \sup_{t}\left\vert Z_{t}\right\vert ^{p}\right] &\leq &C%
\mathbf{E}\left[ \left( \int_{S}^{T}\int_{\left\vert y\right\vert
>1}\left\vert \bar{F}_{r}y\right\vert \frac{dydr}{\left\vert y\right\vert
^{d+\alpha }}\right) ^{p}+\int_{S}^{T}\int_{\left\vert y\right\vert
>1}\left\vert \bar{F}_{r}y\right\vert ^{p}\frac{dydr}{\left\vert
y\right\vert ^{d+\alpha }}\right] \\
&\leq &C\mathbf{E}\left[ \left( \int_{S}^{T}\bar{F}_{r}dr\right)
^{p}+\int_{S}^{T}\bar{F}_{r}^{p}dr\right] \leq C\mathbf{E}\int_{S}^{T}\bar{F}%
_{r}^{p}dr.
\end{eqnarray*}
\end{proof}

We now apply Lemmas \ref{ele1}-\ref{ele12} to estimate 
\begin{equation*}
L_{t}^{0}=\int_{0}^{t}\int_{\left\vert y\right\vert \leq 1}yq\left(
dr,dy\right) ,t\in \lbrack 0,1].
\end{equation*}

\begin{lemma}
\label{c1}Let $0\leq \rho \left( y\right) \leq K$.\newline
(i) There is $C=C\left( \alpha,d ,p,K\right) $ such that for all $t\in \left[
0,1\right] ,$%
\begin{eqnarray}
\mathbf{E}\left[ \left\vert L_{t}^{0}\right\vert ^{p}\right] &\leq &Ct\text{
if }p>\alpha \in \lbrack 1,2),  \notag \\
\mathbf{E}\left[ \left\vert L_{t}^{0}\right\vert ^{p}\right] &\leq &Ct\left(
1+\left\vert \ln t\right\vert \right) \text{ if }p=\alpha \in \lbrack 1,2), 
\notag \\
\mathbf{E}\left[ \left\vert L_{t}^{0}\right\vert ^{p}\right] &\leq
&Ct^{p/\alpha }\text{ if }p<\alpha \in \left( 1,2\right) ,  \notag
\end{eqnarray}%
and%
\begin{equation*}
\mathbf{E}\left[ \left\vert L_{t}^{0}\right\vert ^{p}\right] \leq
Ct^{p}\left( 1+\left\vert \ln t\right\vert \right) ^{p},p<\alpha =1.
\end{equation*}%
(ii) Let $\alpha =1$ and $\rho \left( y\right) =\rho \left( -y\right) ,y\in 
\mathbf{R}^{d}$. There is $C=C(d,p,K)$ such that for all $t\in \left[ 0,1%
\right] ,$%
\begin{equation*}
\mathbf{E}\left[ \left\vert L_{t}^{0}\right\vert ^{p}\right] \leq Ct^{p}%
\text{ if }p<\alpha =1.
\end{equation*}
\end{lemma}

\begin{proof}
These estimates are obvious consequences of Lemmas \ref{ele1} - \ref{ele12}
when they are applied to $F_{r}\left( y\right) =y,y\in \mathbf{R}^{d}.$
\end{proof}

Now we estimate 
\begin{equation*}
L_{t}=L_{t}^{0}+\int_{0}^{t}\int_{\left\vert y\right\vert >1}yN\left(
dr,dy\right) -1_{\alpha \in (1,2)}t\int_{\left\vert y\right\vert >1}y\rho
\left( y\right) \frac{dy}{\left\vert y\right\vert ^{d+\alpha }},
\end{equation*}%
$t\in \left[ S,T\right]$.

\begin{lemma}
\label{c2}Let $0\leq \rho \left( y\right) \leq K$.\newline
(i) For each $p\in \left( 0,\alpha \right) $ there is $C=C\left( \alpha
,d,p,K\right) $ such that for all $t\in \left[ 0,1\right] ,$%
\begin{equation*}
\mathbf{E}\left[ \left\vert L_{t}\right\vert ^{p}\right] \leq Ct^{p/\alpha }%
\text{ if }\alpha \in \left( 1,2\right) ,
\end{equation*}%
and%
\begin{equation*}
\mathbf{E}\left[ \left\vert L_{t}\right\vert ^{p}\right] \leq Ct^{p}\left(
1+\left\vert \ln t\right\vert \right) ^{p}\text{ if }\alpha =1.
\end{equation*}%
(ii) If $\alpha =1$, $\rho \left( y\right) =\rho \left( -y\right) ,y\in 
\mathbf{R}^{d}$. Then for each $p\in \left( 0,\alpha \right) $ there is $%
C=C(p,d,K)$ such that 
\begin{equation*}
\mathbf{E}\left[ \left\vert L_{t}\right\vert ^{p}\right] \leq Ct^{p},t\in %
\left[ 0,1\right] .
\end{equation*}%
(iii) If $\alpha \in \left[ 1,2\right) ,$ then there is $C=C\left( \alpha
,d,K\right) $ such that%
\begin{equation*}
\mathbf{E}\left[ \left\vert L_{t}\right\vert ^{\alpha }\wedge 1\right] \leq
Ct(1+\left\vert \ln t\right\vert ),t\in \left[ 0,1\right] .
\end{equation*}
\end{lemma}

\begin{proof}
The estimates in (i)-(ii) are obvious consequences of Lemmas \ref{c1} and %
\ref{ele2} applied to $F_{r}\left( y\right) =y,y\in \mathbf{R}^{d}$. We
prove (iii) only.

Let 
\begin{eqnarray*}
V_{t} &=&\int_{0}^{t}\int_{\left\vert y\right\vert >1}yN\left( dr,dy\right) ,
\\
P_{t} &=&1_{\alpha \in (1,2)}\int_{0}^{t}\int_{\left\vert y\right\vert
>1}y\rho \left( y\right) \frac{dydr}{\left\vert y\right\vert ^{d+\alpha }},
\end{eqnarray*}%
i.e., $L_{t}=L_{t}^{0}+V_{t}-P_{t},t\in \left[ 0,1\right] .$ According to
Lemma \ref{c1}, there is $C=C\left( \alpha,d ,K\right) $ so that%
\begin{equation*}
\mathbf{E}\left[ \left\vert L_{t}^{0}\right\vert ^{\alpha }\right] \leq
Ct(1+\left\vert \ln t\right\vert ),t\in \left[ 0,1\right]
\end{equation*}%
Now,%
\begin{eqnarray*}
\left\vert V_{t}\right\vert ^{\alpha }\wedge 1
&=&\int_{0}^{t}\int_{\left\vert y\right\vert >1}\left[ \left( \left\vert
V_{r-}+y\right\vert ^{\alpha }\wedge 1\right) -\left( \left\vert
V_{r-}\right\vert ^{\alpha }\wedge 1\right) \right] N\left( dr,dy\right) \\
&\leq &C\int_{0}^{t}\int_{\left\vert y\right\vert >1}\left( \left\vert
y\right\vert \wedge 1\right) N\left( dr,dy\right) ,t\in \left[ 0,1\right] .
\end{eqnarray*}%
Hence%
\begin{equation*}
\mathbf{E}\left[ \left\vert V_{t}\right\vert ^{\alpha }\wedge 1\right] \leq
Ct,t\in \left[ 0,1\right] .
\end{equation*}%
Obviously, $\left\vert P_{t}\right\vert \leq Ct,t\in \left[ 0,1\right] $.
Hence (iii) holds.
\end{proof}

A straightforward consequence of Lemma \ref{c2} is the following statement.

\begin{corollary}
\label{co1}Let $\alpha \in \lbrack 1,2),$~$0\leq \rho \left( y\right) \leq K$%
, $\left\vert b\right\vert \leq K$, $\Vert G\Vert \leq K$.\newline
(i) For each $p\in (0,\alpha ),$ there is $C=C\left( \alpha ,K,d,p\right) $
such that for all $t\in \lbrack 0,1),$ 
\begin{equation*}
\mathbf{E}\left[ \left\vert X_{t}^{n}-X_{\pi _{n}\left( t\right)
}^{n}\right\vert ^{p}\right] \leq Cn^{-p/\alpha }\text{ if }\alpha \in
\left( 1,2\right) ,
\end{equation*}%
and%
\begin{equation*}
\mathbf{E}\left[ \left\vert X_{t}^{n}-X_{\pi _{n}\left( t\right)
}^{n}\right\vert ^{p}\right] \leq C\left( n/\ln n\right) ^{-p}\text{ if }%
\alpha =1.
\end{equation*}%
(ii) If $\alpha =1$, $\rho \left( y\right) =\rho \left( -y\right) ,y\in 
\mathbf{R}^{d}$. Then for each $p\in\left( 0,\alpha\right)$ there is $%
C=C\left( p,d,K\right) $ such that for all $t\in \left[ 0,1\right) $ 
\begin{equation*}
\mathbf{E}\left[ \left\vert X_{t}^{n}-X_{\pi _{n}\left( t\right)
}^{n}\right\vert ^{p}\right] \leq Cn^{-p} .
\end{equation*}
(iii) There is $C=C\left( \alpha ,K,d\right) $ such that for all $t\in
\lbrack 0,1),$ 
\begin{equation*}
\mathbf{E}\left[ \left\vert X_{t}^{n}-X_{\pi _{n}\left( t\right)
}^{n}\right\vert ^{\alpha }\wedge 1\right] \leq C\left( n/\ln n\right) ^{-1}.
\end{equation*}
\end{corollary}

\begin{proof}
For $\forall t\in \lbrack 0,1)$, there is $j\in \left\{ 0,1,\ldots
,n-1\right\} $ so that $j/n\leq t<(j+1)/n$, and $\pi _{n}\left( t\right)
:=j/n$. Thus $0\leq t-\pi _{n}\left( t\right) \leq 1/n$. Note that for any $%
S,t>0$, $L_{t}=L_{S+t}-L_{S}$ in distribution. All the estimates immediately
follow from Lemma \ref{c2}.
\end{proof}

Finally, applying Lemma \ref{c1} we derive

\begin{corollary}
\label{co2}Let $0\leq \rho \left( y\right) \leq K$.\newline
(i) There is $C=C\left( \alpha ,d,p,K\right) $ such that for all $t\in \left[
0,1\right) ,$%
\begin{equation*}
\mathbf{E}\left[ \left\vert Y_{t}^{n}-Y_{\pi _{n}\left( t\right)
}^{n}\right\vert ^{p}\right] \leq C\left\{ 
\begin{array}{cc}
n^{-1} & \text{if }p>\alpha \in \lbrack 1,2), \\ 
\left( n/\ln n\right) ^{-1} & \text{if }p=\alpha \in \lbrack 1,2), \\ 
n^{-p/\alpha } & \text{if }p<\alpha \in \left( 1,2\right) ,%
\end{array}%
\right.
\end{equation*}%
and%
\begin{equation*}
\mathbf{E}\left[ \left\vert Y_{t}^{n}-Y_{\pi _{n}\left( t\right)
}^{n}\right\vert ^{p}\right] \leq C\left( n/\ln n\right) ^{-p}\text{ if }%
p<\alpha =1.
\end{equation*}%
(ii) Let $\alpha =1$, $\rho \left( y\right) =\rho \left( -y\right) ,y\in 
\mathbf{R}^{d}$. There is $C=C(p,d,K)$ such that for all $t\in \left[
0,1\right) ,$%
\begin{equation*}
\mathbf{E}\left[ \left\vert Y_{t}^{n}-Y_{\pi _{n}\left( t\right)
}^{n}\right\vert ^{p}\right] \leq Cn^{-p}\text{ if }p<\alpha =1.
\end{equation*}
\end{corollary}

\begin{proof}
For $\forall t\in \lbrack 0,1)$, there is $j\in \left\{ 0,1,\ldots
,n-1\right\} $ so that $j/n\leq t<(j+1)/n$, and $\pi _{n}\left( t\right)
:=j/n$. Thus $0\leq t-\pi _{n}\left( t\right) \leq 1/n$. Note that for any $%
S,t>0$, $L_{t}^{0}=L_{S+t}^{0}-L_{S}^{0}$ in distribution. All the estimates
immediately follow from Lemma \ref{c1}.
\end{proof}

\section{Proof of Main Results}

We start with the Lipshitz, possibly completely degenerate, case and derive
the rate of convergence directly.

\subsection{Proof of Proposition \protect\ref{pro3}}

Note that 
\begin{equation*}
L_{t}=L_{t}^{0}+V_{t}-1_{\alpha \in \left( 1,2\right) }t\int_{\left\vert
y\right\vert >1}y\rho \left( y\right) \frac{dy}{\left\vert y\right\vert
^{d+\alpha }},t\in \left[ 0,1\right] ,
\end{equation*}%
where%
\begin{equation*}
V_{t}=\int_{0}^{t}\int_{\left\vert y\right\vert >1}yN\left( dr,dy\right)
,t\in \left[ 0,1\right] .
\end{equation*}%
Denote 
\begin{equation*}
\tilde{b}\left( x\right) =b\left( x\right) -1_{\alpha \in \left( 1,2\right)
}G\left( x\right) \int_{\left\vert y\right\vert >1}y\rho \left( y\right)%
\frac{dy}{\left\vert y\right\vert ^{d+\alpha }},x\in \mathbf{R}^{d}.
\end{equation*}

Let $X_{t}$ be the strong solution to (\ref{m1}) and $\bar{X}%
_{t}^{n}:=X_{t}^{n}-X_{t},t\in \left[ 0,1\right] $. Let $0\leq S\leq T\leq
1. $ Then%
\begin{eqnarray}
\bar{X}_{t}^{n} &=&\bar{X}_{S}^{n}+\int_{S}^{t}[\tilde{b}\left( X_{\pi
_{n}\left( r\right) }^{n}\right) -\tilde{b}\left( X_{r}^{n}\right)
]dr+\int_{S}^{t}[\tilde{b}\left( X_{r}^{n}\right) -\tilde{b}\left(
X_{r}\right) ]dr  \notag \\
&&+\int_{S}^{t}[G\left( X_{\pi _{n}\left( r\right) }^{n}\right) -G\left(
X_{r-}^{n}\right) ]dL_{r}^{0}+\int_{S}^{t}[G\left( X_{r-}^{n}\right)
-G\left( X_{r-}^{n}\right) ]dL_{r}^{0}  \notag \\
&&+\int_{S}^{t}[G\left( X_{\pi _{n}\left( r\right) }^{n}\right) -G\left(
X_{r-}^{n}\right) ]dV_{r}+\int_{S}^{t}[G\left( X_{r-}^{n}\right) -G\left(
X_{r-}^{n}\right) ]dV_{r}  \notag \\
&:=&\bar{X}_{S}^{n}+\sum_{k=1}^{6}A_{t}^{k},t\in \left[ S,T\right] .  \notag
\end{eqnarray}

\emph{Estimates of }$A_t^{1}$\emph{\ and}{\footnotesize \ }$A_t^{2}$. For $%
p\in \lbrack 1,\alpha ),\alpha \in \left( 1,2\right) ,$ by H\"{o}lder
inequality, 
\begin{eqnarray*}
\mathbf{E}\left[ \sup_{S\leq t\leq T}\left\vert
A_{t}^{1}+A_{t}^{2}\right\vert ^{p}\right] &\leq &C\mathbf{E}\left[ \left(
\int_{S}^{T}[\left\vert X_{\pi _{n}\left( r\right)
}^{n}-X_{r}^{n}\right\vert \wedge 1]dr\right) ^{p}+\left(
\int_{S}^{T}\left\vert \bar{X}_{r}^{n}\right\vert dr\right) ^{p}\right] \\
&\leq &C\mathbf{E}\left[ \int_{S}^{T}[\left\vert X_{\pi _{n}\left( r\right)
}^{n}-X_{r}^{n}\right\vert ^{p}\wedge 1]dr+\left( T-S\right) ^{p}\sup_{S\leq
t\leq T}\left\vert \bar{X}_{t}^{n}\right\vert ^{p}\right] ,
\end{eqnarray*}%
and for $p\in \left( 0,1\right) ,\alpha \in \lbrack 1,2),$ 
\begin{eqnarray*}
\mathbf{E}\left[ \sup_{S\leq t\leq T}\left\vert
A_{t}^{1}+A_{t}^{2}\right\vert ^{p}\right] &\leq &C\left( \mathbf{E}%
\int_{S}^{T}[\left\vert X_{\pi _{n}\left( r\right)
}^{n}-X_{r}^{n}\right\vert \wedge 1]dr\right) ^{p} \\
&&+C(T-S)^{p}\mathbf{E}\left[ \sup_{S\leq t\leq T}\left\vert \bar{X}%
_{t}^{n}\right\vert ^{p}\right]
\end{eqnarray*}%
for some $C=C\left(\alpha,d, K,p\right) $. By Corollary \ref{co1}, for $p\in
\left( 0,\alpha \right) $ there is $C=C\left(\alpha,d, K,p\right) $ such that%
\begin{equation*}
\mathbf{E}\left[ \sup_{S\leq t\leq T}\left\vert
A_{t}^{1}+A_{t}^{2}\right\vert ^{p}\right] \leq C [l\left( n\right)
^{-p/\alpha }+\left( T-S\right) ^{p}\mathbf{E}[\sup_{S\leq t\leq
T}\left\vert \bar{X}_{t}^{n}\right\vert ^{p}],
\end{equation*}%
where $l\left( n\right) =n$ if $p\in (0,\alpha ),\alpha \in \left(
1,2\right) $, and $l\left( n\right) =n/\ln n$ if $0<p<\alpha =1.$

\emph{Estimate of }$A_t^{3}.$ By definition, 
\begin{equation*}
A_{t}^{3}=\int_{S}^{t}\int_{\left\vert y\right\vert \leq 1}\left[ G\left(
X_{\pi _{n}\left( r\right) }^{n}\right) -G\left( X_{r-}^{n}\right) \right]
yq\left( dr,dy\right) ,t\in \left[ S,T\right] .
\end{equation*}%
According to Corollary \ref{co1}, there is $C=C\left( \alpha,d ,K\right) $
so that 
\begin{eqnarray*}
R :=\mathbf{E}\int_{S}^{T}[\left\vert X_{\pi _{n}\left( r\right)
}^{n}-X_{r}^{n}\right\vert ^{\alpha }\wedge 1]dr\leq C\left( n/\ln n\right)
^{-1}.
\end{eqnarray*}

Apply Lemma \ref{ele11} with $\bar{F}_{r}=\left\vert G\left( X_{\pi
_{n}\left( r\right) }^{n}\right) -G\left( X_{r-}^{n}\right) \right\vert
,r\in \left[ 0,1\right] $, then for all $p\in \left( 0,\alpha \right)
,\alpha \in \left( 1,2\right) ,$ there is $C=C\left( \alpha ,d,K\right) $
such that 
\begin{equation*}
\mathbf{E}\left[ \sup_{S\leq t\leq T}\left\vert A_{t}^{3}\right\vert ^{p} %
\right] \leq C R^{p/\alpha }=C \left( n/\ln n\right) ^{-p/\alpha }.
\end{equation*}

If $\alpha =1,p\in \left( 0,1\right) $, by Lemma \ref{ele11} and H\"{o}lder
inequality, 
\begin{eqnarray*}
\mathbf{E}\left[ \sup_{S\leq t\leq T}\left\vert A_{t}^{3}\right\vert ^{p}%
\right]&\leq &C R^{p}\left( 1+\left\vert\ln R\right\vert\right) ^{p} \leq C
\left( n/\ln n\right) ^{-p}[1+\ln \left( n/\ln n\right) ]^p \\
& \leq& C \left[ n/\left( \ln n\right) ^{2}\right] ^{-p}.
\end{eqnarray*}

If $\alpha =1,p\in \left( 0,1\right), $ and $\rho \left( y\right) =\rho
\left( -y\right) ,y\in \mathbf{R}^{d}$, then applying Lemma \ref{ele12} with 
$H_{r}= G\left( X_{\pi_{n}\left( r\right) }^{n}\right) -G\left(
X_{r-}^{n}\right) ,M=0$, we have 
\begin{equation*}
\mathbf{E}\left[ \sup_{S\leq t\leq T}\left\vert A_{t}^{3}\right\vert ^{p}%
\right] \leq C\left( n/\ln n\right) ^{-p}.
\end{equation*}

\emph{Estimate of }$A_t^{4}.$ By definition, 
\begin{equation*}
A_{t}^{4}=\int_{S}^{t}\int_{\left\vert y\right\vert \leq 1}\left[ G\left(
X_{r-}^{n}\right) -G\left( X_{r-}\right) \right] yq\left( dr,dy\right) ,t\in %
\left[ S,T\right] .
\end{equation*}%
According to Remark \ref{re1}, for $p\in \left( 0,2\right) ,$ there is $%
C=C\left( p,K\right) $ so that%
\begin{equation*}
\mathbf{E}\left[ \sup_{S\leq t\leq T}\left\vert A_{t}^{4}\right\vert ^{p}%
\right] \leq C\mathbf{E}\left[ \left( \int_{S}^{T}\left\vert \bar{X}%
_{r}^{n}\right\vert ^{2}dr\right) ^{p/2}\right] \leq C\left( T-S\right)
^{p/2}\mathbf{E}\left[ \sup_{S\leq t\leq T}\left\vert \bar{X}%
_{t}^{n}\right\vert ^{p}\right] .
\end{equation*}

\emph{Estimate of }$A_t^{5}.$ By definition,%
\begin{equation*}
A_{t}^{5}=\int_{S}^{t}\int_{\left\vert y\right\vert >1}[G\left( X_{\pi
_{n}\left( r\right) }^{n}\right) -G\left( X_{r-}^{n}\right) ]yN\left(
dr,dy\right) ,t\in \left[ S,T\right] .
\end{equation*}

Applying Lemma \ref{ele2} with $F_r\left(y\right) =[G\left( X_{\pi
_{n}\left( r\right) }^{n}\right) -G\left( X_{r-}^{n}\right) ]y$, $\bar{F}%
_{r}=\big\vert G\left( X_{\pi_{n}\left( r\right) }^{n}\right)$ $-G\left(
X_{r-}^{n}\right) \big\vert$, $r\in \left[ 0,1\right] ,y\in \mathbf{R}^{d}$,
and combining 
\begin{equation*}
\bar{F}_{r}=2K\left( \left\vert X_{\pi _{n}\left( r\right)
}^{n}-X_{r}^{n}\right\vert \wedge 1\right) ,r\in \left[ 0,1\right] ,
\end{equation*}
we can conclude for $p\in \left( 0,\alpha \right) $ that 
\begin{equation*}
\mathbf{E}\left[ \sup_{S\leq t\leq T}\left\vert A_{t}^{5}\right\vert ^{p}%
\right] \leq C\mathbf{E}\int_{S}^{T}(\left\vert X_{\pi _{n}\left( r\right)
}^{n}-X_{r}^{n}\right\vert ^{p}\wedge 1)dr.
\end{equation*}%
Hence by Corollary \ref{co1},%
\begin{eqnarray*}
\mathbf{E}\left[ \sup_{S\leq t\leq T}\left\vert A_{t}^{5}\right\vert ^{p}%
\right] &\leq &C\left( n/\ln n\right) ^{-p}\text{ if }p<\alpha =1, \\
\mathbf{E}\left[ \sup_{S\leq t\leq T}\left\vert A_{t}^{5}\right\vert ^{p}%
\right] &\leq &Cn^{-p/\alpha }\text{ if }0<p<\alpha \in \left( 1,2\right) .
\end{eqnarray*}

\emph{Estimate of }$A_t^{6}.$ By definition, 
\begin{equation*}
A_{t}^{6}=\int_{S}^{t}\int_{\left\vert y\right\vert >1}\left[ G\left(
X_{r-}^{n}\right) -G\left( X_{r-}\right) \right] yN\left( dr,dy\right) ,t\in %
\left[ S,T\right] .
\end{equation*}%
By Lemma \ref{ele5} (ii), for $p\in \lbrack 1,\alpha ),\alpha \in \left(
1,2\right) \,$\ there is $C=C\left( \alpha,d ,p,K\right) $ so that%
\begin{eqnarray*}
\mathbf{E}\left[ \sup_{S\leq t\leq T}\left\vert A_{t}^{6}\right\vert ^{p}%
\right] &\leq &C\mathbf{E}\left[ \left( \int_{S}^{T}\left\vert \bar{X}%
_{r}^{n}\right\vert dr\right) ^{p}+\int_{S}^{T}\left\vert \bar{X}%
_{r}^{n}\right\vert ^{p}dr\right] \\
&\leq &C\mathbf{E}\int_{S}^{T}\left\vert \bar{X}_{r}^{n}\right\vert
^{p}dr\leq C\left( T-S\right) \mathbf{E}\left[ \sup_{S\leq t\leq
T}\left\vert \bar{X}_{t}^{n}\right\vert ^{p}\right] .
\end{eqnarray*}%
According to Remark \ref{re3}, for $p\in \left( 0,1\right) \,\ $there is $%
C=C\left( \alpha,d ,p,K\right) $ such that%
\begin{equation*}
\mathbf{E}\left[ \sup_{S\leq t\leq T}\left\vert A_{t}^{6}\right\vert ^{p}%
\right] \leq C\left( T-S\right) \mathbf{E}\left[ \sup_{S\leq t\leq
T}\left\vert \bar{X}_{t}^{n}\right\vert ^{p}\right] .
\end{equation*}

Summarizing, for $p\in \left( 0,\alpha \right) $ there is $C=C\left(
\alpha,d ,p,K\right) $ so that for any $S\leq t\leq T\leq 1,$%
\begin{equation*}
\mathbf{E}\left[ \sup_{S\leq t\leq T}\left\vert \bar{X}_{t}^{n}\right\vert
^{p}\right] \leq C\left\{ \left( n/\ln n\right) ^{-p/\alpha }+\left(
T-S\right) ^{p/2}\mathbf{E}\left[ \sup_{S\leq t\leq T}\left\vert \bar{X}%
_{t}^{n}\right\vert ^{p}\right] \right\}
\end{equation*}%
if $\alpha \in \left( 1,2\right) $, and%
\begin{equation*}
\mathbf{E}\left[ \sup_{S\leq t\leq T}\left\vert \bar{X}_{t}^{n}\right\vert
^{p}\right] \leq C\left\{ \left[ n/\left( \ln n\right) ^{2}\right]
^{-p}+\left( T-S\right) ^{p/2}\mathbf{E}\left[ \sup_{S\leq t\leq
T}\left\vert \bar{X}_{t}^{n}\right\vert ^{p}\right] \right\}
\end{equation*}%
if $\alpha =1$. If $\alpha =1$, and $\rho \left( y\right) =\rho \left(
-y\right) ,y\in \mathbf{R}^{d}$, then 
\begin{equation*}
\mathbf{E}\left[ \sup_{S\leq t\leq T}\left\vert \bar{X}_{t}^{n}\right\vert
^{p}\right] \leq C\left\{ \left( n/\ln n\right) ^{-p}+\left( T-S\right)
^{p/2}\mathbf{E}\left[ \sup_{S\leq t\leq T}\left\vert \bar{X}%
_{t}^{n}\right\vert ^{p}\right] \right\} .
\end{equation*}%
If $C\left( T-S\right) ^{p/2}\leq 1/2,$ then there is $\tilde{C}=\tilde{C}%
\left( \alpha,d ,p,K\right) $ such that for $p\in \left( 0,\alpha \right)$, 
\begin{eqnarray}
\mathbf{E}\left[ \sup_{S\leq t\leq T}\left\vert \bar{X}_{t}^{n}\right\vert
^{p}\right] &\leq &\tilde{C}\left( n/\ln n\right) ^{-p/\alpha }\text{ if }%
\alpha \in \left( 1,2\right) ,  \notag \\
\mathbf{E}\left[ \sup_{S\leq t\leq T}\left\vert \bar{X}_{t}^{n}\right\vert
^{p}\right] &\leq &\tilde{C}\left[ n/\left( \ln n\right) ^{2}\right] ^{-p}%
\text{ if }\alpha =1,  \label{for3} \\
\mathbf{E}\left[ \sup_{S\leq t\leq T}\left\vert \bar{X}_{t}^{n}\right\vert
^{p}\right] &\leq &\tilde{C}\left( n/\ln n\right) ^{-p}\text{ if }\alpha =1 
\text{ with symmetry }.  \notag
\end{eqnarray}

The claim now follows by Lemma \ref{gle}.

\subsection{Proof of Proposition \protect\ref{pro4}}

Let $Y_{t}$ be the strong solution to (\ref{m1'}) and $\bar{Y}%
_{t}^{n}=Y_{t}^{n}-Y_{t},t\in \left[ 0,1\right] $. Let $0\leq S\leq T\leq 1.$
Then%
\begin{eqnarray*}
\bar{Y}_{t}^{n} &=&\bar{Y}_{S}^{n}+\int_{S}^{t}[b\left( Y_{\pi _{n}\left(
r\right) }^{n}\right) -b\left( Y_{r}^{n}\right) ]dr+\int_{S}^{t}[b\left(
Y_{r}^{n}\right) -b\left( Y_{r}\right) ]dr \\
&&+\int_{S}^{t}[G\left( Y_{\pi _{n}\left( r\right) }^{n}\right) -G\left(
Y_{r-}^{n}\right) ]dL_{r}^{0}+\int_{S}^{t}[G\left( Y_{r-}^{n}\right)
-G\left( Y_{r-}\right) ]dL_{r}^{0} \\
&:=&\bar{Y}_{S}^{n}+B_{t}^{1}+B_{t}^{2}+B_{t}^{3}+B_{t}^{4},t\in \left[ S,T%
\right] .
\end{eqnarray*}

\emph{Estimates for \thinspace }$p\in \left( 0,\alpha \right) $ of $%
B^{k},k=1,\ldots ,4,$ are identical to estimates of $A^{k},k=1,\ldots ,4$,
and the conclusion (\ref{for3}) holds for $p\in \left( 0,\alpha \right) $
with $\bar{X}^{n}$ replaced by $\bar{Y}^{n}$.

\emph{Estimates of }$B_t^{1}$\emph{\ and}{\footnotesize \ }$B_t^{2}$\emph{\
for \thinspace }$p\in \lbrack \alpha ,\infty )$. By H\"{o}lder inequality
and Corollary \ref{co2},%
\begin{eqnarray*}
\mathbf{E}[\sup_{S\leq t\leq T}\left\vert B_{t}^{1}\right\vert ^{\alpha}]
&\leq &C\mathbf{E}\int_{S}^{T}\left\vert Y_{\pi _{n}\left( r\right)
}^{n}-Y_{r}^{n}\right\vert ^{\alpha}dr\leq C\left( n/\ln n\right) ^{-1} , \\
\mathbf{E[}\sup_{S\leq t\leq T}\left\vert B_{t}^{1}\right\vert ^{p}] &\leq &C%
\mathbf{E}\int_{S}^{T}\left\vert Y_{\pi _{n}\left( r\right)
}^{n}-Y_{r}^{n}\right\vert ^{p}dr\leq Cn^{-1}\text{ if }p>\alpha .
\end{eqnarray*}%
By H\"{o}lder inequality, for $p\in \lbrack \alpha ,\infty )$ there is $%
C=C\left( K\right) $ such that%
\begin{equation*}
\mathbf{E[}\sup_{S\leq t\leq T}\left\vert B_{t}^{2}\right\vert ^{p}]\leq C%
\mathbf{E}\int_{S}^{T}\left\vert \bar{Y}_{r}^{n}\right\vert ^{p}dr\leq
C\left( T-S\right) \mathbf{E[}\sup_{S\leq t\leq T}\left\vert \bar{Y}%
_{t}^{n}\right\vert ^{p}].
\end{equation*}

\emph{Estimate of }$B_t^{3}$ for $p\in \lbrack \alpha ,\infty )$. By
definition,%
\begin{equation*}
B_{t}^{3}=\int_{S}^{t}\int_{\left\vert y\right\vert \leq 1}\left[ G\left(
Y_{\pi _{n}\left( r\right) }^{n}\right) -G\left( Y_{r-}^{n}\right) \right]
yq\left( dr,dy\right) ,t\in \left[ S,T\right] .
\end{equation*}%
By Corollary \ref{co2}, there is $C=C\left( \alpha,d ,K\right) $ so that 
\begin{equation*}
R:=\mathbf{E}\int_{S}^{T}\left\vert Y_{\pi _{n}\left( r\right)
}^{n}-Y_{r}^{n}\right\vert ^{\alpha }dr\leq C\left( n/\ln n\right) ^{-1}.
\end{equation*}%
Applying Lemma \ref{ele11}(ii) with $\bar{F}_{r}=\left\vert G\left( Y_{\pi
_{n}\left( r\right) }^{n}\right) -G\left( Y_{r-}^{n}\right) \right\vert
,r\in \left[ 0,1\right] $, we can claim there is $C=C\left( \alpha
,d,K\right) $ such that%
\begin{eqnarray*}
\mathbf{E}\left[ \sup_{S\leq t\leq T}\left\vert B_{t}^{3}\right\vert ^{p}%
\right] &\leq &CR\left( 1+\ln R\right) \leq C\left( n/\ln n\right)
^{-1}[1+\ln \left( n/\ln n\right) ] \\
&\leq &C\left[ n/\left( \ln n\right) ^{2}\right] ^{-1}.
\end{eqnarray*}%
Now, for $p>\alpha ,$ by Lemma \ref{ele1} and Corollary \ref{co2}, there is $%
C=C\left( d,p,\alpha ,K\right) $ such that%
\begin{equation*}
\mathbf{E}\left[ \sup_{S\leq t\leq T}\left\vert B_{t}^{3}\right\vert ^{p}%
\right] \leq C\mathbf{E}\int_{S}^{T}\left\vert \bar{Y}_{r}^{n}\right\vert
^{p}dr\leq Cn^{-1}.
\end{equation*}

\emph{Estimate of }$B_t^{4}$ for $p\in \lbrack \alpha ,\infty )$. By
definition,%
\begin{equation*}
B_{t}^{4}=\int_{S}^{t}\int_{\left\vert y\right\vert \leq 1}\left[ G\left(
Y_{r-}^{n}\right) -G\left( Y_{r-}\right) \right] yq\left( dr,dy\right) ,t\in %
\left[ S,T\right] .
\end{equation*}%
By Lemma \ref{ele5}(i) (Kunita's inequality) and Remark \ref{re1}, there is $%
C=C\left( \alpha,d ,p,K\right) $ such that 
\begin{eqnarray*}
\mathbf{E}\left[ \sup_{S\leq t\leq T}\left\vert B_{t}^{4}\right\vert ^{p}%
\right] &\leq &C\mathbf{E}\left[ \left( \int_{S}^{T}\left\vert \bar{Y}%
_{r}^{n}\right\vert ^{2}dr\right) ^{p/2}+\int_{S}^{T}\left\vert \bar{Y}%
_{r}^{n}\right\vert ^{p}dr\right] \\
&\leq &C\left[ \left( T-S\right) +(T-S)^{p/2}\right] \mathbf{E}\left[
\sup_{S\leq t\leq T}\left\vert \bar{Y}_{t}^{n}\right\vert ^{p}\right] .
\end{eqnarray*}

Summarizing, there is $C=C\left( \alpha,d ,p,K\right) $ so that for any $%
S\leq t\leq T\leq 1,$%
\begin{equation*}
\mathbf{E}\left[ \sup_{S\leq t\leq T}\left\vert \bar{Y}_{t}^{n}\right\vert
^{\alpha }\right] \leq C\left\{ \left[ n/\left( \ln n\right) ^{2}\right]
^{-1}+\left( T-S\right) ^{\alpha /2}\mathbf{E}\left[ \sup_{S\leq t\leq
T}\left\vert \bar{Y}_{t}^{n}\right\vert ^{\alpha }\right] \right\} ,
\end{equation*}%
and for all $p>\alpha $ 
\begin{equation*}
\mathbf{E}\left[ \sup_{S\leq t\leq T}\left\vert \bar{X}_{t}^{n}\right\vert
^{p}\right] \leq C\left\{ n^{-1}+\left( T-S\right) ^{\alpha /2}\mathbf{E}%
\left[ \sup_{S\leq t\leq T}\left\vert \bar{Y}_{t}^{n}\right\vert ^{p}\right]
\right\} .
\end{equation*}

We finish the proof by taking $C\left( T-S\right) \leq 1/2$ and applying
Lemma \ref{gle}.

\subsection{Proof of Proposition \protect\ref{pro2}}

First we prove that the Euler approximation sequence is a Cauchy sequence.

\begin{lemma}
\label{lem6}Let $\alpha \in \lbrack 1,2)$, $\beta \in (0,1)$, $\beta
>1-\alpha /2$, $p\in \left( 0,\alpha \right) $ and \textbf{S}$\left(
c_{0}\right) ,$\textbf{A}$(K,c_{0})$ hold. Assume, without loss of
generality, $\left\vert \rho \right\vert _{\beta }\leq K$, $\left\vert
b\right\vert _{\beta }\leq K$ for the same $K$. Then there are constants $%
C_{1}=C_{1}\left( \alpha ,\beta ,d,K,c_{0},p\right) ,c_{1}=c_{1}\left(
\alpha ,\beta ,d,K,c_{0},p\right) $ such that for any $0\leq S\leq T\leq 1$
with $T-S\leq c_{1}$ we have%
\begin{equation*}
\mathbf{E}\left[ \sup_{S\leq t\leq T}\left\vert
X_{t}^{n}-X_{t}^{m}\right\vert ^{p}\right] \leq C_{1}\left( \mathbf{E}\left[
\left\vert X_{S}^{n}-X_{S}^{m}\right\vert ^{p}\right] +n^{-p\beta /\alpha
}+m^{-p\beta /\alpha }\right) .
\end{equation*}

Moreover, if $X_{t}$ is a strong solution to (\ref{m1}), then \ \ \ \ \ \ \
\ \ \ \ \ \ \ \ \ \ \ \ \ \ \ \ \ \ \ \ \ \ \ \ \ \ \ \ \ \ \ \ \ \ \ \ \ \
\ \ \ \ \ \ \ \ \ \ \ \ \ \ \ \ \ \ \ \ \ \ \ \ \ \ \ \ \ \ \ 
\begin{equation*}
\mathbf{E}\left[ \sup_{S\leq t\leq T}\left\vert X_{t}^{n}-X_{t}\right\vert
^{p}\right] \leq C_{1}\left( \mathbf{E}\left[ \left\vert
X_{S}^{n}-X_{S}\right\vert ^{p}\right] +n^{-p\beta /\alpha }\right) .
\end{equation*}
\end{lemma}

\begin{proof}
By Corollary \ref{lem5}, for each $k=1,\ldots ,d$, there exists a unique
solution $u^{k}\left( t,x\right) $ to (\ref{m3}) with%
\begin{equation*}
\tilde{b}\left( x\right) =b\left( x\right) -1_{\alpha \in \left( 1,2\right)
}G\left( x\right) \int_{\left\vert y\right\vert >1}y\rho \left( y\right) 
\frac{dy}{\left\vert y\right\vert ^{d+\alpha }},x\in \mathbf{R}^{d}.
\end{equation*}%
Note that $\tilde{b}$ is also a bounded $\beta $-H\"{o}lder continuous
function. Denote $u=\left( u^{k}\right) _{1\leq k\leq d}$. By It\^{o}
formula and definition of Euler approximation (\ref{m1'}), for $t\in \left[
S,T\right] $, using (\ref{m3}), 
\begin{eqnarray}
&&u^{k}\left( t,X_{t}^{n}\right) -u^{k}\left( S,X_{S}^{n}\right)  \notag \\
&=&\int_{S}^{t}\tilde{b}^{k}\left( X_{r}^{n}\right) dr+\int_{S}^{t}\left[ 
\tilde{b}^{k}\left( X_{\pi _{n}\left( r\right) }^{n}\right) -\tilde{b}%
^{k}\left( X_{r}^{n}\right) \right] \cdot \nabla u^{k}\left(
r,X_{r}^{n}\right) dr  \notag \\
&+&\int_{S}^{t}\int_{\left\vert y\right\vert \leq 1}\left[ u^{k}\left(
r,X_{r-}^{n}+G\left( X_{\pi _{n}\left( r\right) }^{n}\right) y\right)
-u^{k}\left( r,X_{r-}^{n}\right) \right] q\left( dr,dy\right)  \notag \\
&+&\int_{S}^{t}\int_{\left\vert y\right\vert >1}\left[ u^{k}\left(
r,X_{r-}^{n}+G\left( X_{\pi _{n}\left( r\right) }^{n}\right) y\right)
-u^{k}\left( r,X_{r-}^{n}\right) \right] N\left( dr,dy\right)  \notag \\
&+&\int_{S}^{t}\int_{\left\vert y\right\vert \leq 1}\{u^{k}\left(
r,X_{r}^{n}+G\left( X_{\pi _{n}\left( r\right) }^{n}\right) y\right)
-u^{k}\left( r,X_{r}^{n}+G\left( X_{r}^{n}\right) y\right)  \notag \\
&&\quad -\nabla u^{k}\left( r,X_{r}^{n}\right) \cdot \left[ G\left( X_{\pi
_{n}\left( r\right) }^{n}\right) -G\left( X_{r-}^{n}\right) \right] y\}\rho
\left( y\right) \frac{dydr}{\left\vert y\right\vert ^{d+\alpha }}.  \notag
\end{eqnarray}%
On the other hand, according to \eqref{m2}, for $t\in \left[ S,T\right] $ 
\begin{eqnarray*}
&&X_{t}^{n}-X_{S}^{n} \\
&=&\int_{S}^{t}\tilde{b}\left( X_{r}^{n}\right) dr+\int_{S}^{t}\left[ \tilde{%
b}\left( X_{\pi _{n}\left( r\right) }^{n}\right) -\tilde{b}\left(
X_{r}^{n}\right) \right] dr \\
&+&\int_{S}^{t}\int_{\left\vert y\right\vert >1}G\left( X_{\pi _{n}\left(
r\right) }^{n}\right) yN\left( dr,dy\right) +\int_{S}^{t}\int_{\left\vert
y\right\vert \leq 1}G\left( X_{\pi _{n}\left( r\right) }^{n}\right) yq\left(
dr,dy\right) .
\end{eqnarray*}%
It follows from the two identities above that%
\begin{equation*}
X_{t}^{n}=\sum_{k=1}^{7}D_{t}^{n,k},
\end{equation*}%
where%
\begin{eqnarray*}
D_{t}^{n,1} &=&X_{S}^{n}+\big[u\left( t,X_{t}^{n}\right) -u\left(
S,X_{S}^{n}\right) \big], \\
D_{t}^{n,2} &=&\int_{S}^{t}\left[ \tilde{b}\left( X_{\pi _{n}\left( r\right)
}^{n}\right) -\tilde{b}\left( X_{r}^{n}\right) \right] \big(I_{d}-\nabla
u\left( r,X_{r}^{n}\right) \big)dr, \\
D_{t}^{n,3} &=&\int_{S}^{t}\int_{\left\vert y\right\vert \leq 1}\{u\left(
r,X_{r}^{n}+G\left( X_{r}^{n}\right) y\right) -u\left( r,X_{r}^{n}+G\left(
X_{\pi _{n}\left( r\right) }^{n}\right) y\right) \\
&&\quad -\nabla u\left( r,X_{r}^{n}\right) \cdot \left[ G\left(
X_{r-}^{n}\right) -G\left( X_{\pi _{n}\left( r\right) }^{n}\right) \right]
y\}\rho \left( y\right) \frac{dydr}{\left\vert y\right\vert ^{d+\alpha }}, \\
D_{t}^{n,4} &=&\int_{S}^{t}\int_{\left\vert y\right\vert \leq 1}\Big\{%
\lbrack u\left( r,X_{r-}^{n}+G\left( X_{r-}^{n}\right) y\right) -u\left(
r,X_{r-}^{n}+G\left( X_{\pi _{n}\left( r\right) }^{n}\right) y\right) ] \\
&&\quad +[G\left( X_{\pi _{n}\left( r\right) }^{n}\right) -G\left(
X_{r-}^{n}\right) ]y\Big\}q\left( dr,dy\right) , \\
D_{t}^{n,5} &=&\int_{S}^{t}\int_{\left\vert y\right\vert >1}[u\left(
r,X_{r-}^{n}+G\left( X_{r-}^{n}\right) y\right) -u\left(
r,X_{r-}^{n}+G\left( X_{\pi _{n}\left( r\right) }^{n}\right) y\right) ] \\
&&\quad +[G\left( X_{\pi _{n}\left( r\right) }^{n}\right) -G\left(
X_{r-}^{n}\right) ]y\Big\}N\left( dr,dy\right) , \\
D_{t}^{n,6} &=&\int_{S}^{t}\int_{\left\vert y\right\vert \leq 1}\Big\{%
G\left( X_{r-}^{n}\right) y-\left[ u\left( r,X_{r-}^{n}+G\left(
X_{r-}^{n}\right) y\right) -u\left( r,X_{r-}^{n}\right) \right] \Big\}%
q\left( dr,dy\right) , \\
D_{t}^{n,7} &=&\int_{S}^{t}\int_{\left\vert y\right\vert >1}\Big\{G\left(
X_{r-}^{n}\right) y-\left[ u\left( r,X_{r-}^{n}+G\left( X_{r-}^{n}\right)
y\right) -u\left( r,X_{r-}^{n}\right) \right] \Big\}N\left( dr,dy\right) .
\end{eqnarray*}

Let $D_t^{n,m;k}=D_t^{n,k}-D_t^{m,k},$ and $X_t^{n,m}=X_t^{n}-X_t^{m},n,m%
\geq 1,k=1,\ldots ,7$.

\emph{Estimate of }$D_{t}^{n,m;1}.$ Using the terminal condition of (\ref{m3}%
) and Corollary \ref{lem5}, we see that for $p\in \left( 0,\infty \right) $
there is a constant $C=C\left( \alpha ,\beta ,d,K,p,c_{0}\right) $ such that 
\begin{eqnarray*}
\left\vert D_{t}^{n,m;1}\right\vert ^{p} &\leq &C\{\left\vert
X_{S}^{m}-X_{S}^{n}\right\vert ^{p}+\left\vert u\left( t,X_{t}^{n}\right)
-u\left( t,X_{t}^{m}\right) +u\left( T,X_{t}^{m}\right) -u\left(
T,X_{t}^{n}\right) \right\vert ^{p} \\
&&+\left\vert u\left( T,X_{S}^{n}\right) -u\left( T,X_{S}^{m}\right)
+u\left( S,X_{S}^{m}\right) -u\left( S,X_{S}^{n}\right) \right\vert ^{p}\} \\
&\leq &C\{\left[ \left\vert X_{S}^{n}-X_{S}^{m}\right\vert ^{p}\right]
+\left( T-t\right) ^{p/2}\left\vert X_{t}^{n}-X_{t}^{m}\right\vert ^{p}\},
\end{eqnarray*}%
therefore, 
\begin{equation*}
\mathbf{E}\left[ \sup_{S\leq t\leq T}\left\vert D_{t}^{n,m;1}\right\vert ^{p}%
\right] \leq C\{\left( T-S\right) ^{p/2}\mathbf{E}\left[ \sup_{S\leq t\leq
T}\left\vert X_{t}^{n,m}\right\vert ^{p}\right] +\mathbf{E}\left\vert
X_{S}^{n,m}\right\vert ^{p}\}.
\end{equation*}

\emph{Estimate of }$D_{t}^{n,m;2}$. Obviously, $\left\vert
D_{t}^{n,m;2}\right\vert ^{p}\leq 2^{p}[\left\vert D_{t}^{n,2}\right\vert
^{p}+\left\vert D_{t}^{m,2}\right\vert ^{p}],t\in \left[ 0,1\right) .$

For $p\in \lbrack 1,\alpha )$, $\alpha \in \left( 1,2\right) ,$ by H\"{o}%
lder inequality and Corollary \ref{co1}, there is $C=C\left( \alpha ,\beta
,d,c_{0},K,p\right) $ such that%
\begin{equation*}
\mathbf{E}\left[ \sup_{S\leq t\leq T}\left\vert D_{t}^{n,2}\right\vert ^{p}%
\right] \leq C\mathbf{E}\left[ \int_{S}^{T}\left\vert X_{\pi _{n}\left(
r\right) }^{n}-X_{r}^{n}\right\vert ^{p\beta }dr\right] \leq Cn^{-\beta
p/\alpha }.
\end{equation*}

For $p\in \left( 0,1\right) $, by Corollary \ref{co1}, there is a constant $%
C=C\left( \alpha ,\beta ,d,K,p\right) $ such that 
\begin{eqnarray*}
\mathbf{E}\left[ \sup_{S\leq t\leq T}\left\vert D_{t}^{n,2}\right\vert ^{p}%
\right] &\leq &C\mathbf{E}\left[ \left( \int_{S}^{T}\left\vert X_{\pi
_{n}\left( r\right) }^{n}-X_{r}^{n}\right\vert ^{\beta }dr\right) ^{p}\right]
\\
&\leq &C\left( \int_{S}^{T}\mathbf{E}[\left\vert X_{\pi _{n}\left( r\right)
}^{n}-X_{r}^{n}\right\vert ^{\beta }]dr\right) ^{p}\leq Cn^{-\beta p/\alpha
}.
\end{eqnarray*}

Similarly, we can obtain the estimates for $D_t^{m,2}$. Hence, by H\"{o}lder
inequality, for all ~$p\in (0,\alpha)$, there is $C=C\left( \alpha ,\beta
,d,K,p\right) $ such that 
\begin{equation*}
\mathbf{E}\left[ \sup_{S\leq t\leq T}\left\vert D_{t}^{n,m;2}\right\vert ^{p}%
\right] \leq C[n^{-p\beta /\alpha }+m^{-p\beta /\alpha }].
\end{equation*}

\emph{Estimate of }$D_{t}^{n,m;3}$. Obviously, $\left\vert
D_{t}^{n,m;3}\right\vert ^{p}\leq 2^{p}[\left\vert D_{t}^{n,3}\right\vert
^{p}+\left\vert D_{t}^{m,3}\right\vert ^{p}],t\in \left[ 0,1\right) $. Note
that 
\begin{eqnarray*}
\left\vert D_{t}^{n,3}\right\vert &=&\Big\vert\int_{S}^{t}\int_{\left\vert
y\right\vert \leq 1}\Big\{\lbrack \int_{0}^{1}[\nabla u\left(
r,X_{r}^{n}\right) -\nabla u\Big(r,X_{r}^{n}+G\left( X_{r}^{n}\right) y+s%
\Big\lbrack G\left( X_{\pi _{n}\left( r\right) }^{n}\right) \\
&&-G\left( X_{r}^{n}\right) \Big\rbrack y\Big)]ds\cdot \left[ G\left( X_{\pi
_{n}\left( r\right) }^{n}\right) -G\left( X_{r-}^{n}\right) \right] y\Big\}%
\rho \left( y\right) \frac{dydr}{\left\vert y\right\vert ^{d+\alpha }}%
\Big\vert.
\end{eqnarray*}

Let $\beta ^{\prime }\in \left( 0,1\right) ,\alpha +\beta >1+\beta ^{\prime
}>\alpha $ and denote $G_{r}^{n}=\left\vert G\left( X_{\pi _{n}\left(
r\right) }^{n}\right) -G\left( X_{r-}^{n}\right) \right\vert ,r\in \left[
0,1\right) $. Then there is $C=C\left( \alpha ,d,K,c_{0},\beta \right) $
such that%
\begin{eqnarray*}
\left\vert D_{t}^{n,3}\right\vert &\leq &C\int_{S}^{T}\int_{\left\vert
y\right\vert \leq 1}G_{r}^{n}\left\vert y\right\vert ^{1+\beta ^{\prime }}%
\frac{dy}{\left\vert y\right\vert ^{d+\alpha }}dr \\
&\leq &C\int_{S}^{T}\left\vert X_{r}^{n}-X_{\pi _{n}\left( r\right)
}^{n}\right\vert \wedge 1dr,\quad t\in \left[ S,T\right] .
\end{eqnarray*}%
Hence by Corollary \ref{co1}, for $p\in \left[ 1,\alpha \right) ,\alpha \in
\left( 1,2\right) $, 
\begin{equation*}
\mathbf{E}\left[ \sup_{S\leq t\leq T}\left\vert D_{t}^{n,3}\right\vert ^{p}%
\right] \leq C\int_{S}^{T}\mathbf{E}\left[ \left\vert X_{r}^{n}-X_{\pi
_{n}\left( r\right) }^{n}\right\vert ^{p}\right] dr\leq Cn^{-p/\alpha };
\end{equation*}%
for $p\in \left( 0,1\right) ,$ according to Corollary \ref{co1}, 
\begin{equation*}
\mathbf{E}\left[ \sup_{S\leq t\leq T}\left\vert D_{t}^{n,3}\right\vert ^{p}%
\right] \leq C\left( \int_{S}^{T}\mathbf{E}\left[ \left\vert
X_{r}^{n}-X_{\pi _{n}\left( r\right) }^{n}\right\vert \wedge 1\right]
dr\right) ^{p}\leq C\left( n/\ln n\right) ^{-p/\alpha }.
\end{equation*}

Similar reasoning can be applied to $\left\vert D_{t}^{m,3}\right\vert $.
Therefore for all $p\in \left( 0,\alpha \right) $ there is $C=C\left(
\alpha,d ,p,K,\beta\right) $ such that 
\begin{equation*}
\mathbf{E}\left[ \sup_{S\leq t\leq T}\left\vert D_{t}^{n,m;3}\right\vert ^{p}%
\right] \leq C\left[ \left( n/\ln n\right) ^{-p/\alpha }+\left( m/\ln
m\right) ^{-p/\alpha }\right] .
\end{equation*}

\emph{Estimate of }$D_{t}^{n,m;4}$. Obviously, $\left\vert
D_{t}^{n,m;4}\right\vert ^{p}\leq 2^{p}[\left\vert D_{t}^{n,4}\right\vert
^{p}+\left\vert D_{t}^{m,4}\right\vert ^{p}],t\in \left[ 0,1\right) $. By
Corollary \ref{co1}(iii), there is $C=C\left( \alpha,d ,K\right) $ such that%
\begin{equation*}
R:=\mathbf{E}\int_{S}^{T}\left\vert X_{\pi _{n}\left( r\right)
}^{n}-X_{r-}^{n}\right\vert ^{\alpha }\wedge 1dr\leq C\left( n/\ln n\right)
^{-1}.
\end{equation*}

First, let $\alpha \in \left( 1,2\right) $. Applying Lemma \ref{ele11}(i) to 
$D_{t}^{n,4}$ with 
\begin{equation}
\bar{F}_{r}=2K\left( 1+|\nabla u|_{0}\right) \left( \left\vert X_{\pi
_{n}\left( r\right) }^{n}-X_{r-}^{n}\right\vert \wedge 1\right) ,r\in \left[
S,T\right] ,  \label{for4}
\end{equation}%
and Corollary \ref{co1}(iii), we have that for $p\in \left( 0,\alpha \right) 
$ there is $C=C\left( \alpha ,d,p,K,\beta ,c_{0}\right) $ such that 
\begin{equation*}
\mathbf{E}\left[ \sup_{S\leq t\leq T}\left\vert D_{t}^{n,4}\right\vert ^{p}%
\right] \leq CR^{p/\alpha }\leq C\left( n/\ln n\right) ^{-p/\alpha }.
\end{equation*}

Now, let $\alpha =1$. Applying Lemma \ref{ele12}(i) to $D_t^{n,4}$ with $%
\bar{ F}_{r}$ given by (\ref{for4}), and Corollary \ref{co1}(iii), we see
there is $C=C\left(\alpha,d, p,K,\beta\right) $ such that%
\begin{equation*}
\mathbf{E}\left[ \sup_{S\leq t\leq T}\left\vert D_{t}^{n,4}\right\vert ^{p}%
\right] \leq CR^{p}\left( 1+\left\vert\ln R\right\vert\right) ^{p}\leq C%
\left[ n/\left( \ln n\right) ^{2}\right] ^{-p}.
\end{equation*}

Similarly, 
\begin{equation*}
\mathbf{E}\left[ \sup_{S\leq t\leq T}\left\vert D_{t}^{m,4}\right\vert ^{p}%
\right] \leq C\left[ m/\left( \ln m\right) ^{2}\right] ^{-p/\alpha },
\end{equation*}%
and thus there is $C=C\left( \alpha ,d,p,K,\beta ,c_{0}\right) $ so that%
\begin{equation*}
\mathbf{E}\left[ \sup_{S\leq t\leq T}\left\vert D_{t}^{n,m;4}\right\vert ^{p}%
\right] \leq C\left\{ \left[ n/\left( \ln n\right) ^{2}\right] ^{-p/\alpha }+%
\left[ m/\left( \ln m\right) ^{2}\right] ^{-p/\alpha }\right\} .
\end{equation*}

\emph{Estimate of }$D_{t}^{n,m;5}$. Obviously, $\left\vert
D_{t}^{n,m;5}\right\vert ^{p}\leq 2^{p}[\left\vert D_{t}^{n,5}\right\vert
^{p}+\left\vert D_{t}^{m,5}\right\vert ^{p}],t\in \left[ 0,1\right) $. By
Lemma \ref{ele2}, applied to $D_{t}^{n,5}$ with 
\begin{equation*}
\bar{F}_{r}=\left( 1+|\nabla u|_{0}\right) \left\vert \nabla G\right\vert
_{\infty }\left\vert X_{\pi _{n}\left( r\right) }^{n}-X_{r-}^{n}\right\vert
,r\in \left[ S,T\right] ,
\end{equation*}%
and Corollary \ref{co1}, there is $C=C\left( \alpha ,d,p,K,\beta
,c_{0}\right) $ such that 
\begin{equation*}
\mathbf{E}\left[ \sup_{S\leq t\leq T}\left\vert D_{t}^{n,5}\right\vert ^{p}%
\right] \leq Cn^{-p/\alpha }\text{ for }p\in \left( 0,\alpha \right) .
\end{equation*}

Similarly as above, for $p\in\left( 0,\alpha\right)$, 
\begin{equation*}
\mathbf{E}\left[ \sup_{S\leq t\leq T}\left\vert D_{t}^{n,m;5}\right\vert ^{p}%
\right] \leq C\left[ n^{-p/\alpha }+m^{-p/\alpha }\right] .
\end{equation*}

\emph{Estimate of }$D_{t}^{n,m;6}$. Denote $G_{r}^{n,m}=G\left(
X_{r}^{n}\right) -G\left( X_{r}^{m}\right) ,r\in \left[ S,T\right] $. Then 
\begin{eqnarray*}
D_{t}^{n,m;6} &=&\int_{S}^{t}\int_{\left\vert y\right\vert \leq 1}\Big\{%
\lbrack G\left( X_{r-}^{n}\right) -G\left( X_{r-}^{m}\right) ]y \\
&&\quad -\left[ u\left( r,X_{r-}^{n}+G\left( X_{r-}^{n}\right) y\right)
-u\left( r,X_{r-}^{n}+G\left( X_{r-}^{m}\right) y\right) \right] \Big\}%
q\left( dr,dy\right) \\
&-&\int_{S}^{t}\int_{\left\vert y\right\vert \leq 1}\{[u\left(
r,X_{r-}^{n}+G\left( X_{r-}^{m}\right) y\right) -u\left(
r,X_{r-}^{m}+G\left( X_{r-}^{m}\right) y\right) ] \\
&&\quad +\left[ u\left( r,X_{r}^{m}\right) -u\left( r,X_{r}^{n}\right) %
\right] \}q\left( dr,dy\right) \\
&:=&D_{t}^{n,m;61}+D_{t}^{n,m;62}.
\end{eqnarray*}%
For $p\in \left( 0,2\right) $, by Remark \ref{re1}, there is $C=C\left(
\alpha ,\beta ,d,K,c_{0},p\right) $ such that%
\begin{eqnarray*}
\mathbf{E}\left[ \sup_{S\leq t\leq T}\left\vert D_{t}^{n,m;61}\right\vert
^{p}\right] &\leq &C\mathbf{E}\left[ \left( \int_{S}^{T}\int_{\left\vert
y\right\vert \leq 1}\left\vert G_{r}^{n,m}y\right\vert ^{2}\frac{dydr}{%
\left\vert y\right\vert ^{d+\alpha }}\right) ^{p/2}\right] \\
&\leq &C\left( T-S\right) ^{p/2}\mathbf{E}\left[ \sup_{S\leq r\leq
T}\left\vert X_{r}^{n,m}\right\vert ^{p}\right] .
\end{eqnarray*}

We rewrite 
\begin{eqnarray*}
D_{t}^{n,m;62} &=&\int_{S}^{t}\int_{\left\vert y\right\vert \leq
1}\int_{0}^{1}[-\nabla u\left( r,X_{r-}^{m}+G\left( X_{r-}^{m}\right)
y+sX_{r-}^{n,m}\right) \\
&&+\nabla u\left( r,X_{r}^{m}+sX_{r-}^{n,m}\right) ]X_{r-}^{n,m}dsq\left(
dr,dy\right) ,t\in \left[ S,T\right] .
\end{eqnarray*}%
Let $1+\beta ^{\prime }<\alpha +\beta $ and $2\beta ^{\prime }>\alpha $.
Then by Remark \ref{re1}, there is $C=C\left( \alpha ,\beta
,K,p,c_{0},d\right) $ such that 
\begin{eqnarray*}
\mathbf{E}\left[ \sup_{S\leq t\leq T}\left\vert D_{t}^{n,m;62}\right\vert
^{p}\right] &\leq &C\mathbf{E}\left[ \left( \int_{S}^{T}\left\vert
X_{r}^{n,m}\right\vert ^{2}dr\int_{\left\vert y\right\vert \leq 1}\left\vert
y\right\vert ^{2\beta ^{\prime }}\frac{dy}{\left\vert y\right\vert
^{d+\alpha }}\right) ^{p/2}\right] \\
&\leq &C\left( T-S\right) ^{p/2}\mathbf{E}\left[ \sup_{S\leq t\leq
T}\left\vert X_{r}^{n,m}\right\vert ^{p}\right] .
\end{eqnarray*}%
Hence, 
\begin{equation*}
\mathbf{E}\left[ \sup_{S\leq t\leq T}\left\vert D_{t}^{n,m;6}\right\vert ^{p}%
\right] \leq C\left( T-S\right) ^{p/2}\mathbf{E}\left[ \sup_{S\leq t\leq
T}\left\vert X_{r}^{n,m}\right\vert ^{p}\right] .
\end{equation*}

\emph{Estimate of }$D_{t}^{n,m;7}$. Let $\alpha \in \left( 1,2\right) ,p\in
\lbrack 1,\alpha )$. By Lemma \ref{ele5}(ii) (see Lemma 4.1 in \cite{lm}),
there is $C=C\left( \alpha ,\beta ,K,c_{0},p,d\right) $ such that 
\begin{eqnarray*}
\mathbf{E}\left[ \sup_{S\leq t\leq T}\left\vert D_{t}^{n,m;7}\right\vert ^{p}%
\right] &\leq &C\mathbf{E}\Big[\left( \int_{S}^{T}\int_{\left\vert
y\right\vert >1}[\left\vert G_{r}^{n,m}y\right\vert +\left\vert
X_{r}^{n,m}\right\vert ]\frac{dydr}{\left\vert y\right\vert ^{d+\alpha }}%
\right) ^{p} \\
&&+\int_{S}^{T}\int_{\left\vert y\right\vert >1}[\left\vert
G_{r}^{n,m}y\right\vert ^{p}+\left\vert X_{r}^{n,m}\right\vert ^{p}]\frac{%
dydr}{\left\vert y\right\vert ^{d+\alpha }}\Big] \\
&\leq &C\left( T-S\right) \mathbf{E}\left[ \sup_{S\leq t\leq T}\left\vert
X_{t}^{n,m}\right\vert ^{p}\right] .
\end{eqnarray*}

Let $\alpha =1,p\in \left( 0,1\right) $. By Remark \ref{re3}, there is $%
C=C\left( \alpha ,\beta ,K,p,d\right) $ such that%
\begin{equation*}
\left\vert D_{t}^{n,m;7}\right\vert ^{p}\leq C\int_{S}^{T}\int_{\left\vert
y\right\vert >1}[\left\vert G_{r-}^{n,m}y\right\vert +\left\vert
X_{r-}^{n,m}\right\vert ]^{p}N\left( dr,dy\right) ,t\in \left[ S,T\right] ,
\end{equation*}%
and thus 
\begin{eqnarray*}
\mathbf{E}\left[ \sup_{S\leq t\leq T}\left\vert D_{t}^{n,m;7}\right\vert ^{p}%
\right] &\leq &C\mathbf{E}\Big[\int_{S}^{T}\int_{\left\vert y\right\vert
>1}[\left\vert G_{r}^{n,m}y\right\vert ^{p}+\left\vert
X_{r}^{n,m}\right\vert ^{p}]\frac{dydr}{\left\vert y\right\vert ^{d+\alpha }}%
\Big] \\
&\leq &C\left( T-S\right) \mathbf{E}\left[ \sup_{S\leq r\leq T}\left\vert
X_{r}^{n,m}\right\vert ^{p}\right] .
\end{eqnarray*}

Collecting all the estimates above we see that for $p\in \left( 0,\alpha
\right) $ there is $C=C\left( \alpha ,\beta ,K,p,d\right) $ such that 
\begin{eqnarray}
\mathbf{E}\left[ \sup_{S\leq t\leq T}\left\vert X_{t}^{n,m}\right\vert ^{p}%
\right] &\leq &C\Big\{\left( T-S\right) ^{p/2}\mathbf{E}\left[ \sup_{S\leq
t\leq T}\left\vert X_{t}^{n,m}\right\vert ^{p}\right] +\mathbf{E}\left[
\left\vert X_{S}^{n,m}\right\vert ^{p}\right]  \notag \\
&&+n^{-p\beta /\alpha }+m^{-p\beta /\alpha }\Big\}.  \label{est}
\end{eqnarray}

Set $c_{1}=\left( 2C\right) ^{-2/p},C_{1}=2C$ with the $C$ in (\ref{est}),
we then have 
\begin{equation*}
\mathbf{E}\left[ \sup_{S\leq t\leq T}\left\vert X_{t}^{n,m}\right\vert ^{p}%
\right] \leq C_{1}\Big\{\mathbf{E}\left[ \left\vert X_{S}^{n,m}\right\vert
^{p}\right] +n^{-p\beta /\alpha }+m^{-p\beta /\alpha }\Big\}
\end{equation*}%
if $0\leq T-S\leq c_{1}$.

\emph{Rate of convergence. }Now let us assume $X_{t}$ is a strong solution
to (\ref{m1}). We have, by It\^{o} formula and \eqref{m3}, for $t\in \left[
S,T\right] $, 
\begin{eqnarray*}
&&u\left( t,X_{t}\right) -u\left( S,X_{S}\right) \\
&=&\int_{S}^{t}\tilde{b}\left( X_{r}\right) dr+\int_{S}^{t}\int_{\left\vert
y\right\vert \leq 1}\left[ u\left( r,X_{r-}+G\left( X_{r-}\right) y\right)
-u\left( r,X_{r-}\right) \right] q\left( dr,dy\right) \\
&+&\int_{S}^{t}\int_{\left\vert y\right\vert >1}\left[ u\left(
r,X_{r-}+G\left( X_{r-}\right) y\right) -u\left( r,X_{r-}\right) \right]
N\left( dr,dy\right) ,
\end{eqnarray*}%
Hence for $t\in \left[ S,T\right] $, we obtain 
\begin{eqnarray*}
&&X_{t}-X_{S}=u\left( t,X_{t}\right) -u\left( S,X_{S}\right) \\
&+&\int_{S}^{t}\int_{\left\vert y\right\vert \leq 1}\{G\left( X_{r-}\right)
y-\left[ u\left( r,X_{r-}+G\left( X_{r-}\right) y\right) -u\left(
r,X_{r-}\right) \right] \}q\left( dr,dy\right) \\
&+&\int_{S}^{t}\int_{\left\vert y\right\vert >1}\{G\left( X_{r-}\right) y- 
\left[ u\left( r,X_{r-}+G\left( X_{r-}\right) y\right) -u\left(
r,X_{r-}\right) \right] \}N\left( dr,dy\right) ,
\end{eqnarray*}

and 
\begin{eqnarray*}
&&X_{t}^{n}-X_{t} \\
&=&\{X_{S}^{n}-X_{S}+[u\left( t,X_{t}^{n}\right) -u\left( S,X_{S}^{n}\right)
]-[u\left( t,X_{t}\right) -u\left( S,X_{S}\right) ]\} \\
&+&\sum_{k=2}^{5}D_{t}^{n,k}+D_{t}^{n,6}+D_{t}^{n,7} \\
&-&\int_{S}^{t}\int_{\left\vert y\right\vert \leq 1}\{G\left( X_{r-}\right)
y-\left[ u\left( r,X_{r-}+G\left( X_{r-}\right) y\right) -u\left(
r,X_{r-}\right) \right] \}q\left( dr,dy\right) \\
&-&\int_{S}^{t}\int_{\left\vert y\right\vert >1}\{G\left( X_{r-}\right) y- 
\left[ u\left( r,X_{r-}+G\left( X_{r-}\right) y\right) -u\left(
r,X_{r-}\right) \right] \}N\left( dr,dy\right) .
\end{eqnarray*}

Estimates for $D_{t}^{n,k},k=2,\ldots ,5$ have been derived above. And we
can estimate 
\begin{eqnarray*}
&&\qquad \qquad X_{S}^{n}-X_{S}+[u\left( t,X_{t}^{n}\right) -u\left(
S,X_{S}^{n}\right) ]-u\left( t,X_{t}\right) -u\left( S,X_{S}\right) , \\
&&D_{t}^{n,6}-\int_{S}^{t}\int_{\left\vert y\right\vert \leq 1}\{G\left(
X_{r-}\right) y-\left[ u\left( r,X_{r-}+G\left( X_{r-}\right) y\right)
-u\left( r,X_{r-}\right) \right] \}q\left( dr,dy\right) , \\
&&D_{t}^{n,7}-\int_{S}^{t}\int_{\left\vert y\right\vert >1}\{G\left(
X_{r-}\right) y-\left[ u\left( r,X_{r-}+G\left( X_{r-}\right) y\right)
-u\left( r,X_{r-}\right) \right] \}N\left( dr,dy\right)
\end{eqnarray*}%
in exactly the same way as we estimated $D_{t}^{n,m;1},D_{t}^{n,m;6}$ and $%
D_{t}^{n,m;7}$ (by replacing $X_{t}^{m}$ by $X_{t}$ in the arguments). We
find that there is a constant $C=C\left( \alpha ,\beta ,p,K,c_{0},d\right) $
such that%
\begin{equation*}
\mathbf{E}\left[ \sup_{S\leq t\leq T}\left\vert X_{t}^{n}-X_{t}\right\vert
^{p}\right] \leq C_{1}\left( \mathbf{E}\left[ \left\vert
X_{S}^{n}-X_{S}\right\vert ^{p}\right] +n^{-p\beta /\alpha }\right) .
\end{equation*}%
Then the claimed rate of convergence holds because of Lemma \ref{gle}.
\end{proof}

\emph{Existence of a solution.} Let $p\in (0,\alpha )$ and $c_{1}$ be the
constant in Lemma \ref{lem6}. By Lemmas \ref{gle} and \ref{lem6}, there is $%
C=C\left( \alpha ,\beta ,p,K,c_{0},d\right) $ such that for $n,m\geq 1,$%
\begin{equation*}
\mathbf{E}\left[ \sup_{0\leq t\leq 1}\left\vert
X_{t}^{n}-X_{t}^{m}\right\vert ^{p}\right] \leq C\left( n^{-p\beta /\alpha
}+m^{-p\beta /\alpha }\right) ,
\end{equation*}%
and thus 
\begin{equation*}
\mathbf{E}\left[ \sup_{0\leq t\leq 1}\left\vert
X_{t}^{n}-X_{t}^{m}\right\vert ^{p}\right] \rightarrow 0
\end{equation*}%
as $n,m\rightarrow \infty $. Therefore there is an adapted c\`{a}dl\`{a}g
process $X_{t}$ such that for all $p\in \left( 0,\alpha \right) ,$%
\begin{equation*}
\mathbf{E}\left[ \sup_{0\leq t\leq 1}\left\vert X_{t}^{n}-X_{t}\right\vert
^{p}\right] \rightarrow 0
\end{equation*}%
as $n\rightarrow \infty $. Hence $X_{t}$ solves (\ref{m1}). Moreover, by
Lemma \ref{lem6}, there is $C=C\left( \alpha ,\beta ,d,K,p\right) $ such
that 
\begin{equation*}
\mathbf{E}\left[ \sup_{0\leq t\leq 1}\left\vert X_{t}^{n}-X_{t}\right\vert
^{p}\right] \leq Cn^{-p\beta /\alpha }.
\end{equation*}

\emph{Uniqueness }follows from Lemma \ref{lem6}: any strong solution can be
approximated by $X_{t}^{n}$.

\subsection{Proof of Proposition \protect\ref{t1}}

The proof repeats the steps we took to prove Proposition \ref{pro2}.

\begin{lemma}
\label{lem7}Let $\alpha \in \lbrack 1,2)$, $\beta \in (0,1)$, $\beta
>1-\alpha /2$, $p\in \left( 0,\alpha \right) $ and \textbf{S}$\left(
c_{0}\right) ,$\textbf{A}$(K,c_{0})$ hold. Assume (without loss of
generality), $\left\vert \rho \right\vert _{\beta }\leq K$, $\left\vert
b\right\vert _{\beta }\leq K$ for the same $K$. Then there are constants $%
C_{1}=C_{1}\left( \alpha ,\beta ,d,K,c_{0},p\right) ,c_{1}=c_{1}\left(
\alpha ,\beta ,d,K,c_{0},p\right) $ such that for any $0\leq S\leq T\leq 1$
with $T-S\leq c_{1}$ we have%
\begin{equation*}
\mathbf{E}\left[ \sup_{S\leq t\leq T}\left\vert
Y_{t}^{n}-Y_{t}^{m}\right\vert ^{p}\right] \leq C_{1}\left( \mathbf{E}\left[
\left\vert Y_{S}^{n}-Y_{S}^{m}\right\vert ^{p}\right] +l(n,\beta ,\alpha
,p)+l(m,\beta ,\alpha ,p)\right) ,
\end{equation*}%
where $l(k,\beta ,\alpha ,p)=k^{-p\beta /\alpha }$ if $p\beta <\alpha $, $%
l(k,\beta ,\alpha ,p)=\left( k/\ln k\right) ^{-1}$ if $p\beta =\alpha $, and 
$l(k,\beta ,\alpha ,p)=k^{-1}$ if $p\beta >\alpha .$

Moreover, if $Y_{t}$ is a strong solution to (\ref{m2}), then \ \ \ \ \ \ \
\ \ \ \ \ \ \ \ \ \ \ \ \ \ \ \ \ \ \ \ \ \ \ \ \ \ \ \ \ \ \ \ \ \ \ \ \ \
\ \ \ \ \ \ \ \ \ \ \ \ \ \ \ \ \ \ \ \ \ \ \ \ \ \ \ \ \ \ \ 
\begin{equation*}
\mathbf{E}\left[ \sup_{S\leq t\leq T}\left\vert Y_{t}^{n}-Y_{t}\right\vert
^{p}\right] \leq C_{1}l(n,\beta ,\alpha ,p).
\end{equation*}
\end{lemma}

\begin{proof}
Let $0\leq S\leq T\leq 1$. By Corollary \ref{lem5}, for each $k=1,\ldots ,d$%
, there exists a unique solution $u^{k}\left( t,x\right) $ to (\ref{m3})
with $\tilde{b}\left( x\right) =b\left( x\right) ,x\in \mathbf{R}^{d}.$
Denote $u=\left( u^{k}\right) _{1\leq k\leq d}$. By It\^{o} formula and
definition of Euler approximation (\ref{m2'}), for $t\in \left[ S,T\right] $%
, using (\ref{m3}), 
\begin{eqnarray}
&&u^{k}\left( t,Y_{t}^{n}\right) -u^{k}\left( S,Y_{S}^{n}\right)  \notag \\
&=&\int_{S}^{t}b^{k}\left( Y_{r}^{n}\right) dr+\int_{S}^{t}\left[
b^{k}\left( Y_{\pi _{n}\left( r\right) }^{n}\right) -b^{k}\left(
Y_{r}^{n}\right) \right] \cdot \nabla u^{k}\left( r,Y_{r}^{n}\right) dr 
\notag \\
&+&\int_{S}^{t}\int_{\left\vert y\right\vert \leq 1}\left[ u^{k}\left(
r,Y_{r-}^{n}+G\left( Y_{\pi _{n}\left( r\right) }^{n}\right) y\right)
-u^{k}\left( r,Y_{r-}^{n}\right) \right] q\left( dr,dy\right)  \notag \\
&+&\int_{S}^{t}\int_{\left\vert y\right\vert \leq 1}\{u^{k}\left(
r,Y_{r}^{n}+G\left( Y_{\pi _{n}\left( r\right) }^{n}\right) y\right)
-u^{k}\left( r,Y_{r}^{n}+G\left( Y_{r}^{n}\right) y\right)  \notag \\
&&\quad -\nabla u^{k}\left( r,Y_{r}^{n}\right) \cdot \left[ G\left( Y_{\pi
_{n}\left( r\right) }^{n}\right) -G\left( Y_{r}^{n}\right) \right] y\}\rho
\left( y\right) \frac{dydr}{\left\vert y\right\vert ^{d+\alpha }}.  \notag
\end{eqnarray}%
On the other hand, according to (\ref{m2'}), for $t\in \left[ S,T\right] $ 
\begin{eqnarray*}
Y_{t}^{n}-Y_{S}^{n} &=& \int_{S}^{t}b\left( Y_{r}^{n}\right) dr+\int_{S}^{t} 
\left[ b\left( Y_{\pi _{n}\left( r\right) }^{n}\right) -b\left(
Y_{r}^{n}\right) \right] dr \\
&&+\int_{S}^{t}\int_{\left\vert y\right\vert \leq 1}G\left( Y_{\pi
_{n}\left( r\right) }^{n}\right) yq\left( dr,dy\right) .
\end{eqnarray*}%
It follows from the two identities above that%
\begin{equation*}
Y_{t}^{n}=\sum_{k=1}^{4}B_{t}^{n,k}+B_{t}^{n,5},
\end{equation*}%
where%
\begin{eqnarray*}
B_{t}^{n,1} &=&Y_{S}^{n}+\big[u\left( t,Y_{t}^{n}\right) -u\left(
S,Y_{S}^{n}\right) \big], \\
B_{t}^{n,2} &=&\int_{S}^{t}\left[ b\left( Y_{\pi _{n}\left( r\right)
}^{n}\right) -b\left( Y_{r}^{n}\right) \right] \big(I_{d}-\nabla u\left(
r,Y_{r}^{n}\right) \big)dr, \\
B_{t}^{n,3} &=&\int_{S}^{t}\int_{\left\vert y\right\vert \leq 1}\{u\left(
r,Y_{r}^{n}+G\left( Y_{r}^{n}\right) y\right) -u\left( r,Y_{r}^{n}+G\left(
Y_{\pi _{n}\left( r\right) }^{n}\right) y\right) \\
&&\quad -\nabla u\left( r,Y_{r}^{n}\right) \cdot \left[ G\left(
Y_{r}^{n}\right) -G\left( Y_{\pi _{n}\left( r\right) }^{n}\right) \right]
y\}\rho \left( y\right) \frac{dydr}{\left\vert y\right\vert ^{d+\alpha }}, \\
B_{t}^{n,4} &=&\int_{S}^{t}\int_{\left\vert y\right\vert \leq 1}\Big\{%
\lbrack u\left( r,Y_{r-}^{n}+G\left( Y_{r-}^{n}\right) y\right) -u\left(
r,Y_{r-}^{n}+G\left( Y_{\pi _{n}\left( r\right) }^{n}\right) y\right) ] \\
&&\quad +[G\left( Y_{\pi _{n}\left( r\right) }^{n}\right) -G\left(
Y_{r-}^{n}\right) ]y\Big\}q\left( dr,dy\right) , \\
B_{t}^{n,5} &=&\int_{S}^{t}\int_{\left\vert y\right\vert \leq 1}\Big\{%
G\left( Y_{r-}^{n}\right) y-\left[ u\left( r,Y_{r-}^{n}+G\left(
Y_{r-}^{n}\right) y\right) -u\left( r,Y_{r-}^{n}\right) \right] \Big\}%
q\left( dr,dy\right) .
\end{eqnarray*}

Let $B_t^{n,m;k}=B_t^{n,k}-B_t^{m,k},$ and $Y_t^{n,m}=Y_t^{n}-Y_t^{m},n,m%
\geq 1,k=1,\ldots ,5$.

\emph{Estimate of }$B_{t}^{n,m;1}$. This estimate is identical to that of $%
D^{n,m;1}$ in the proof of Lemma \ref{lem6}. Repeating it and applying
Corollary \ref{lem5}, we see that for $p\in \left( 0,\infty \right) $ there
is $C=C\left( \alpha ,\beta ,p,K,c_{0},d\right) $ so that 
\begin{equation*}
\left\vert B_{t}^{n,m;1}\right\vert ^{p}\leq C\{\left[ \left\vert
Y_{S}^{n}-Y_{S}^{m}\right\vert ^{p}\right] +\left( T-t\right)
^{p/2}\left\vert Y_{t}^{n}-Y_{t}^{m}\right\vert ^{p}\},
\end{equation*}%
and, 
\begin{equation*}
\mathbf{E}\left[ \sup_{S\leq t\leq T}\left\vert B_{t}^{n,m;1}\right\vert ^{p}%
\right] \leq C\{\left( T-S\right) ^{p/2}\mathbf{E}\left[ \sup_{S\leq t\leq
T}\left\vert Y_{t}^{n,m}\right\vert ^{p}\right] +\mathbf{E}\left\vert
Y_{S}^{n,m}\right\vert ^{p}\}.
\end{equation*}

Estimates of $B_t^{n,m;k},k=2,3,4,$ for $p\in \left( 0,\alpha \right) $ are
identical to the estimates of $D_t^{n,m;k},k=2,3,4.$ We replace $X$ by $Y$,
and apply Corollary \ref{co2} instead of \ref{co1}. Note that for $p\in
\left( 0,\alpha \right) $ the estimates in Corollary \ref{co1} coincide with
estimates in Corollary \ref{co2}. Hence for $p\in \left( 0,\alpha \right) $
there is $C=C\left( \alpha ,\beta ,p,c_{0},K\right) $ such that for $%
k=2,3,4, $%
\begin{equation*}
\mathbf{E}\left[ \sup_{S\leq t\leq T}\left\vert B_{t}^{n,m;k}\right\vert ^{p}%
\right] \leq C\left( n^{-p\beta /\alpha }+m^{-p\beta /\alpha }\right) .
\end{equation*}

\emph{Estimate of }$B_{t}^{n,m;2}$ \emph{for }$p\in \lbrack \alpha ,\infty )$%
. By H\"{o}lder inequality and Corollary \ref{lem5}, there is $C=C\left(
\alpha ,\beta ,d,K,c_{0},p\right) $ such that%
\begin{equation*}
\mathbf{E}\left[ \sup_{S\leq t\leq T}\left\vert B_{t}^{n,2}\right\vert ^{p}%
\right] \leq C\mathbf{E}\left[ \int_{S}^{T}\left\vert Y_{\pi _{n}\left(
r\right) }^{n}-Y_{r}^{n}\right\vert ^{p\beta }dr\right] .
\end{equation*}%
Hence, by Corollary \ref{co2},%
\begin{equation*}
\mathbf{E}\left[ \sup_{S\leq t\leq T}\left\vert B_{t}^{n,2}\right\vert ^{p}%
\right] \leq Cl(n,\beta ,\alpha ,p),
\end{equation*}%
where $l(n,\beta ,\alpha ,p)=n^{-p\beta /\alpha }$ if $p\beta <\alpha $, $%
l(n,\beta ,\alpha ,p)=\left( n/\ln n\right) ^{-1}$ if $p\beta =\alpha $, and 
$l(n,\beta ,\alpha ,p)=n^{-1}$ if $p\beta >\alpha .$ Therefore for $p\in %
\left[ \alpha ,\infty \right) $, 
\begin{equation*}
\mathbf{E}\left[ \sup_{S\leq t\leq T}\left\vert B_{t}^{n,m;2}\right\vert ^{p}%
\right] \leq C[l(n,\beta ,\alpha ,p)+l\left( m,\beta ,\alpha ,p\right) ].
\end{equation*}

\emph{Estimate of }$B_{t}^{n,m;3}$ \emph{for }$p\in \lbrack \alpha ,\infty )$%
. By repeating the argument for $D_{t}^{n,3}$ in the proof of Proposition %
\ref{pro2}, we find that there is $C=C\left( \alpha ,\beta ,d,K\right) $ so
that 
\begin{equation*}
\left\vert B_{t}^{n,3}\right\vert \leq C\int_{S}^{T}\left\vert
Y_{r}^{n}-Y_{\pi _{n}\left( r\right) }^{n}\right\vert dr,t\in \left[ S,T%
\right] .
\end{equation*}%
Hence by Corollary \ref{co2}, for $p\geq \alpha $ there is $C=C\left( \alpha
,\beta ,d,K,c_{0},p\right) $ such that 
\begin{equation*}
\mathbf{E}\left[ \sup_{S\leq t\leq T}\left\vert B_{t}^{n,3}\right\vert ^{p}%
\right] \leq C\int_{S}^{T}\mathbf{E}\left[ \left\vert Y_{r}^{n}-Y_{\pi
_{n}\left( r\right) }^{n}\right\vert ^{p}\right] dr\leq Cl\left( n,1,\alpha
,p\right) .
\end{equation*}%
Therefore, for $p\in \lbrack \alpha ,\infty )$ there is $C=C\left( \alpha
,\beta ,p,K,c_{0},d\right) $ such that 
\begin{eqnarray*}
\mathbf{E}\left[ \sup_{S\leq t\leq T}\left\vert B_{t}^{n,m;3}\right\vert ^{p}%
\right] &\leq &C\left[ l\left( n,1,\alpha ,p\right) +l\left( m,1,\alpha
,p\right) \right] \\
&\leq &C\left[ l\left( n,\beta ,\alpha ,p\right) +l\left( m,\beta ,\alpha
,p\right) \right] .
\end{eqnarray*}

\emph{Estimate of }$B_{t}^{n,m;4}$ \emph{for }$p\in \lbrack \alpha ,\infty
). $

\begin{eqnarray*}
B_{t}^{n,4} &=&\int_{S}^{t}\int_{\left\vert y\right\vert \leq 1}\Big\{%
\lbrack u\left( r,Y_{r-}^{n}+G\left( Y_{r-}^{n}\right) y\right) -u\left(
r,Y_{r-}^{n}+G\left( Y_{\pi _{n}\left( r\right) }^{n}\right) y\right) ] \\
&&\quad +[G\left( Y_{\pi _{n}\left( r\right) }^{n}\right) -G\left(
Y_{r-}^{n}\right) ]y\Big\}q\left( dr,dy\right) ,
\end{eqnarray*}%
By Corollary \ref{co2}(i), there is $C=C\left( \alpha ,d,K\right) $ such that%
\begin{equation*}
R:=\mathbf{E}\int_{S}^{T}\left\vert Y_{\pi _{n}\left( r\right)
}^{n}-Y_{r}^{n}\right\vert ^{\alpha }dr\leq C\left( n/\ln n\right) ^{-1}.
\end{equation*}

Applying Lemma \ref{ele11}(ii) with 
\begin{equation}
\bar{F}_{r}=\left( 1+|\nabla u|_{0}\right) \left\vert \nabla G\right\vert
_{\infty }\left\vert Y_{\pi _{n}\left( r\right) }^{n}-Y_{r-}^{n}\right\vert
,r\in \left[ S,T\right] ,  \label{e3}
\end{equation}%
we can see there is $C=C\left( \alpha ,\beta ,K,c_{0},d\right) $ such that%
\begin{equation*}
\mathbf{E}\left[ \sup_{S\leq t\leq T}\left\vert B_{t}^{n,4}\right\vert
^{\alpha }\right] \leq CR\left( 1+\left\vert \ln R\right\vert \right) \leq C%
\left[ n/\left( \ln n\right) ^{2}\right] ^{-1}.
\end{equation*}

By Lemma \ref{ele1} with $\bar{F}_{r}$ given by (\ref{e3}) and Corollary \ref%
{co2}, for \thinspace $p>\alpha \,\ $there is $C=C\left( \alpha ,\beta
,p,d,K,c_{0}\right) $ such that%
\begin{equation*}
\mathbf{E}\left[ \sup_{S\leq t\leq T}\left\vert B_{t}^{n,4}\right\vert ^{p}%
\right] \leq Cn^{-1}.
\end{equation*}

Hence for $p\geq \alpha $ there is $C=C\left( \alpha ,\beta
,p,d,K,c_{0}\right) $ such that%
\begin{equation*}
\mathbf{E}\left[ \sup_{S\leq t\leq T}\left\vert B_{t}^{n,m;4}\right\vert ^{p}%
\right] \leq C\left[ l\left( n,\beta ,\alpha ,p\right) +l\left( m,\beta
,\alpha ,p\right) \right] .
\end{equation*}

\emph{Estimate of }$B_t^{n,m;5}.$ As in the case of $D_t^{n,m;6}$ in the
proof of Proposition \ref{pro2}, we rewrite 
\begin{eqnarray*}
B_{t}^{n,m;5} &=&\int_{S}^{t}\int_{\left\vert y\right\vert \leq 1}\Big\{%
\lbrack G\left( Y_{r-}^{n}\right) -G\left( Y_{r-}^{m}\right) ]y \\
&&\quad -\left[ u\left( r,Y_{r-}^{n}+G\left( Y_{r-}^{n}\right) y\right)
-u\left( r,Y_{r-}^{n}+G\left( Y_{r-}^{m}\right) y\right) \right] \Big\}%
q\left( dr,dy\right) \\
&-&\int_{S}^{t}\int_{\left\vert y\right\vert \leq 1}\{[u\left(
r,Y_{r-}^{n}+G\left( Y_{r-}^{m}\right) y\right) -u\left(
r,Y_{r-}^{m}+G\left( Y_{r-}^{m}\right) y\right) ] \\
&&\quad +\left[ u\left( r,Y_{r}^{m}\right) -u\left( r,Y_{r}^{n}\right) %
\right] \}q\left( dr,dy\right) \\
&:=&B_{t}^{n,m;51}+B_{t}^{n,m;52},
\end{eqnarray*}%
and 
\begin{eqnarray*}
B_{t}^{n,m;52} &=&\int_{S}^{t}\int_{\left\vert y\right\vert \leq
1}\int_{0}^{1}[-\nabla u\left( r,Y_{r-}^{m}+G\left( Y_{r-}^{m}\right)
y+sY_{r-}^{n,m}\right) \\
&&+\nabla u\left( r,Y_{r-}^{m}+sY_{r-}^{n,m}\right) ]Y_{r-}^{n,m}dsq\left(
dr,dy\right) ,t\in \left[ S,T\right] .
\end{eqnarray*}

For $p\in \left( 0,2\right) $, repeating the estimates of $D^{m,m;6}$ in the
proof of Proposition \ref{pro2}, we find that for $p\in \left( 0,2\right) $
there is $C=C\left( \alpha ,p,K,c_{0},\beta ,d\right) $ such that%
\begin{equation*}
\mathbf{E}\left[ \sup_{S\leq t\leq T}\left\vert B_{t}^{n,m;5}\right\vert ^{p}%
\right] \leq C\left( T-S\right) ^{p/2}\mathbf{E}\left[ \sup_{S\leq t\leq
T}\left\vert Y_{r}^{n,m}\right\vert ^{p}\right] .
\end{equation*}

For $p\geq 2,$ by Lemma \ref{ele5}(i), there is $C=C\left( \alpha
,p,K,c_{0},\beta ,d\right) $ such that%
\begin{equation*}
\mathbf{E}\left[ \sup_{S\leq t\leq T}\left\vert B_{t}^{n,m;5}\right\vert ^{p}%
\right] \leq C\left( T-S\right) \mathbf{E}\left[ \sup_{S\leq t\leq
T}\left\vert Y_{r}^{n,m}\right\vert ^{p}\right] .
\end{equation*}

Collecting all the estimates above we see that for $p\in \left( 0,\infty
\right) $ there is $C=C\left( \alpha ,\beta ,K,c_{0},p,d\right) $ such that 
\begin{eqnarray*}
\mathbf{E}\left[ \sup_{S\leq t\leq T}\left\vert Y_{t}^{n,m}\right\vert ^{p}%
\right] &\leq &C\Big\{\lbrack \left( T-S\right) ^{p/2}+(T-S)]\mathbf{E}\left[
\sup_{S\leq t\leq T}\left\vert Y_{t}^{n,m}\right\vert ^{p}\right] \\
&&+\mathbf{E}\left[ \left\vert Y_{S}^{n,m}\right\vert ^{p}\right] +l\left(
n,\beta ,\alpha ,p\right) +l\left( m,\beta ,\alpha ,p\right) \Big\}.
\end{eqnarray*}

There is $c_{1}=c_{1}\left( \alpha ,\beta ,K,c_{0},d,p\right) $ such that $C%
\left[ \left( T-S\right) ^{p/2}+\left( T-S\right) \right] \leq 1/2$ if $%
0\leq T-S\leq c_{1}$. In that case%
\begin{equation*}
\mathbf{E}\left[ \sup_{S\leq t\leq T}\left\vert Y_{t}^{n,m}\right\vert ^{p}%
\right] \leq 2C\Big\{\mathbf{E}\left[ \left\vert Y_{S}^{n,m}\right\vert ^{p}%
\right] +l\left( n,\beta ,\alpha ,p\right) +l\left( m,\beta ,\alpha
,p\right) \Big\}.
\end{equation*}

\emph{Rate of convergence. }Now let us assume $Y_{t}$ is a strong solution
to (\ref{m1'}). We have, by It\^{o} formula and (\ref{m3}), for $t\in \left[
S,T\right] $, 
\begin{eqnarray*}
&&u\left( t,Y_{t}\right) -u\left( S,Y_{S}\right) \\
&=&\int_{S}^{t}b\left( Y_{r}\right) dr+\int_{S}^{t}\int_{\left\vert
y\right\vert \leq 1}\left[ u\left( r,Y_{r-}+G\left( Y_{r-}\right) y\right)
-u\left( r,Y_{r-}\right) \right] q\left( dr,dy\right) .
\end{eqnarray*}%
Hence for $t\in \left[ S,T\right] $, we obtain 
\begin{eqnarray*}
&&Y_{t}-Y_{S}=u\left( t,Y_{t}\right) -u\left( S,Y_{S}\right) \\
&+&\int_{S}^{t}\int_{\left\vert y\right\vert \leq 1}\{G\left( Y_{r-}\right)
y-\left[ u\left( r,Y_{r-}+G\left( Y_{r-}\right) y\right) -u\left(
r,Y_{r-}\right) \right] \}q\left( dr,dy\right) ,
\end{eqnarray*}
and thus 
\begin{eqnarray*}
&&Y_{t}^{n}-Y_{t} \\
&=&\{Y_{S}^{n}-Y_{S}+[u\left( t,Y_{t}^{n}\right) -u\left( S,Y_{S}^{n}\right)
]-[u\left( t,Y_{t}\right) -u\left( S,Y_{S}\right) ]\} \\
&+&\sum_{k=2}^{4}B_{t}^{n,k}+B_{t}^{n,5} \\
&-&\int_{S}^{t}\int_{\left\vert y\right\vert \leq 1}\{G\left( Y_{r-}\right)
y-\left[ u\left( r,Y_{r-}+G\left( Y_{r-}\right) y\right) -u\left(
r,Y_{r-}\right) \right] \}q\left( dr,dy\right) .
\end{eqnarray*}

Estimates for $B^{n,k},k=2,3,4$ have been derived above. And we can estimate 
\begin{eqnarray*}
&&\qquad \quad Y_{S}^{n}-Y_{S}+[u\left( t,Y_{t}^{n}\right) -u\left(
S,Y_{S}^{n}\right) ]-u\left( t,Y_{t}\right) -u\left( S,Y_{S}\right) , \\
&&B_{t}^{n,6}-\int_{S}^{t}\int_{\left\vert y\right\vert \leq 1}\{G\left(
Y_{r-}\right) y-\left[ u\left( r,Y_{r-}+G\left( Y_{r-}\right) y\right)
-u\left( r,Y_{r-}\right) \right] \}q\left( dr,dy\right)
\end{eqnarray*}%
in exactly the same way as we estimated $B_{t}^{n,m;1}$ and $B_{t}^{n,m;5}$
(by replacing $Y_{t}^{m}$ by $Y_{t}$ in the arguments). We find that there
is a constant $C=C\left( \alpha ,\beta ,p,K,c_{0},d\right) $ such that%
\begin{equation*}
\mathbf{E}\left[ \sup_{S\leq t\leq T}\left\vert Y_{t}^{n}-Y_{t}\right\vert
^{p}\right] \leq C[\mathbf{E}\left[ \left\vert Y_{S}^{n}-Y_{S}\right\vert
^{p}\right] +l\left( n,\beta ,\alpha ,p\right) ],
\end{equation*}%
and the claimed rate of convergence holds by Lemma \ref{gle}.
\end{proof}

The existence and uniqueness part is a simple repeat of the arguments in the
proof of Proposition \ref{pro2}.

\section{Appendix}

We will be using some general estimates of stochastic integrals. We start
with Lenglart's inequality (see \cite{le}). Let $Z_{t}$ be a nonnegative c%
\`{a}dl\`{a}g process and $A_{t}$ be an increasing predictable process. We
say that $A$ dominates $Z$ if for any finite stopping time $\tau $ 
\begin{equation*}
\mathbf{E}Z_{\tau }\leq \mathbf{E}A_{\tau }.
\end{equation*}

The following moment estimate holds.

\begin{lemma}
\label{lem3}( see Corollary II in \cite{le}) Let $Z$ be dominated by $A$.
Then for every $p\in \left( 0,1\right) $ and every stopping time $\tau $, 
\begin{equation*}
\mathbf{E}\left[ \left( \sup_{s\leq \tau }\left\vert Z_{\tau }\right\vert
\right) ^{p}\right] \leq \frac{2-p}{1-p}\mathbf{E}\left[ A_{\tau }^{p}\right]%
.
\end{equation*}
\end{lemma}

\begin{remark}
\label{re1}Let $H:[0,1)\times \Omega \times \mathbf{R}_{0}^{d}\rightarrow 
\mathbf{R}^{m}$ be a $\mathcal{P\times B}\left( \mathbf{R}_{0}^{d}\right) $%
-measurable function, $H:=H_{r}\left( y\right) ,r\in \lbrack 0,1),y\in 
\mathbf{R}^{d}$. Assume that for any $T\in \lbrack 0,1)$ a.s.,%
\begin{equation*}
\int_{0}^{T}\int \left\vert H_{r}\left( y\right) \right\vert ^{2}\rho \left(
y\right) \frac{dydr}{\left\vert y\right\vert ^{d+\alpha }}<\infty ,
\end{equation*}%
where $\mathcal{P}$ is a predictable $\sigma $-algebra on $[0,1)\times
\Omega $. Then

(i)\ (see \cite{le}) 
\begin{equation*}
Z_{t}=\left\vert \int_{0}^{t}\int H_{r}\left( y\right) q\left( dr,dy\right)
\right\vert ^{2},t\in \lbrack 0,1),
\end{equation*}%
is dominated by%
\begin{equation*}
A_{t}=\int_{0}^{t}\int \left\vert H_{r}\left( y\right) \right\vert ^{2}\rho
\left( y\right) \frac{dy}{\left\vert y\right\vert ^{d+\alpha }}dr,t\in
\lbrack 0,1).
\end{equation*}%
Hence by Lemma \ref{lem3} (Corollary II in \cite{le}), for any $p\in \left(
0,2\right) $ there is $C=C\left( p\right) $ such that for any stopping time $%
\tau ,$%
\begin{eqnarray*}
&&\mathbf{E}\left[ \sup_{t\leq \tau }\left\vert \int_{0}^{t}\int H_{r}\left(
y\right) q\left( dr,dy\right) \right\vert ^{p}\right] \\
&\leq &C\mathbf{E}\left[ \left( \int_{0}^{\tau }\int \left\vert H_{r}\left(
y\right) \right\vert ^{2}\rho \left( y\right) \frac{dy}{\left\vert
y\right\vert ^{d+\alpha }}dr\right) ^{p/2}\right] .
\end{eqnarray*}

(ii)\ On the other hand, for $p\in \left[ 1,2\right] $, by BGD inequality, 
\begin{eqnarray}
\mathbf{E}\left[ \sup_{t\leq \tau }\left\vert \int_{0}^{t}\int H_{r}\left(
y\right) q\left( dr,dy\right) \right\vert ^{p}\right] &\leq &C\mathbf{E}%
\left[ \left( \int_{0}^{\tau }\int \left\vert H_{r}\left( y\right)
\right\vert ^{2}\rho \left( y\right) N\left( dr,dy\right) \right) ^{p/2}%
\right]  \label{ff1} \\
&\leq &C\mathbf{E}\left[ \int_{0}^{\tau }\int \left\vert H_{r}\left(
y\right) \right\vert ^{p}\rho \left( y\right) N\left( dr,dy\right) \right] 
\notag \\
&\leq &C\mathbf{E}\int_{0}^{\tau }\int \left\vert H_{r}\left( y\right)
\right\vert ^{p}\frac{dy}{\left\vert y\right\vert ^{d+\alpha }}dr.  \notag
\end{eqnarray}
\end{remark}

\begin{remark}
\label{re2}Let $H:[0,1)\times \Omega \times \mathbf{R}_{0}^{d}\rightarrow 
\mathbf{R}^{m}$ be a $\mathcal{P\times B}\left( \mathbf{R}_{0}^{d}\right) $%
-measurable function, $H:=H_{r}\left( y\right) ,r\in \lbrack 0,1),y\in 
\mathbf{R}^{d},$ such that for any $T\in \lbrack 0,1)$ a.s.,%
\begin{equation*}
\int_{0}^{T}\int \left\vert H_{r}\left( y\right) \right\vert \rho \left(
y\right) \frac{dydr}{\left\vert y\right\vert ^{d+\alpha }}<\infty .
\end{equation*}%
(i) Obviously, 
\begin{equation*}
Z_{t}=\left\vert \int_{0}^{t}\int H_{r}\left( y\right) q\left( dr,dy\right)
\right\vert ,t\in \lbrack 0,1),
\end{equation*}%
is dominated by%
\begin{equation*}
A_{t}=2\int_{0}^{t}\int \left\vert H_{r}\left( y\right) \right\vert \rho
\left( y\right) \frac{dy}{\left\vert y\right\vert ^{d+\alpha }}dr,t\in
\lbrack 0,1).
\end{equation*}%
Hence by Lemma \ref{lem3} (Corollary II in \cite{le}), for any $p\in \left(
0,1\right) $ there is $C=C\left( p\right) $ such that for any stopping time $%
\tau ,$%
\begin{equation*}
\mathbf{E}\left[ \sup_{t\leq \tau }\left\vert \int_{0}^{t}\int H_{r}\left(
y\right) q\left( dr,dy\right) \right\vert ^{p}\right] \leq C\mathbf{E}\left[
\left( \int_{0}^{\tau }\int \left\vert H_{r}\left( y\right) \right\vert \rho
\left( y\right) \frac{dy}{\left\vert y\right\vert ^{d+\alpha }}dr\right) ^{p}%
\right] .
\end{equation*}%
(ii) For $p\in \lbrack 1,2],$ by BDG\ inequality, we have as in (\ref{ff1}),%
\begin{eqnarray*}
\mathbf{E}\left[ \sup_{t\leq \tau }\left\vert \int_{0}^{t}\int H_{r}\left(
y\right) q\left( dr,dy\right) \right\vert ^{p}\right] &\leq &C\mathbf{E}%
\left[ \left( \int_{0}^{\tau }\int \left\vert H_{r}\left( y\right)
\right\vert ^{2}\rho \left( y\right) N\left( dr,dy\right) \right) ^{p/2}%
\right] \\
&\leq &C\mathbf{E}\int_{0}^{\tau }\int \left\vert H_{r}\left( y\right)
\right\vert ^{p}\frac{dy}{\left\vert y\right\vert ^{d+\alpha }}dr.
\end{eqnarray*}
\end{remark}

For the sake of completeness we remind two other \textquotedblleft
general\textquotedblright estimates.

\begin{lemma}
\label{ele5}(see e.g. Lemma 4.1 in \cite{lm}) (i) (Kunita's inequality) Let $%
H:[0,1)\times \Omega \times \mathbf{R}_{0}^{d}\rightarrow \mathbf{R}^{m}$ be
a $\mathcal{P\times B}\left( \mathbf{R}_{0}^{d}\right) $-measurable
function, $H:=H_{r}\left( y\right) ,r\in \lbrack 0,1),y\in \mathbf{R}^{d},$
such that for any $T\in \lbrack 0,1)$ a.s.,%
\begin{equation*}
\int_{0}^{T}\int \left\vert H_{r}\left( y\right) \right\vert ^{2}\rho \left(
y\right) \frac{dydr}{\left\vert y\right\vert ^{d+\alpha }}<\infty ,
\end{equation*}%
where $\mathcal{P}$ is a predictable $\sigma $-algebra on $[0,1)\times
\Omega $. Then for each $p\geq 2$ there is $C=C\left( p\right) $ such that
for any stopping time $\tau $, 
\begin{align*}
\mathbf{E}\left[ \sup_{t\leq \tau }\left\vert \int_{0}^{t}\int
H_{r}(y)q(dr,dy)\right\vert ^{p}\right] &\leq C\mathbf{E}\left[
\int_{0}^{\tau }\int \left\vert H_{r}(y)\right\vert ^{p}\rho \left( y\right) 
\frac{dy}{\left\vert y\right\vert ^{d+\alpha }}dr\ \right] \\
&\quad +C\mathbf{E}\left[ \left( \int_{0}^{\tau }\int \left\vert
H_{r}(y)\right\vert ^{2}\rho \left( y\right) \frac{dy}{\left\vert
y\right\vert ^{d+\alpha }}dr\right) ^{p/2}\right] .
\end{align*}

(ii) Let $H:[0,1)\times \Omega \times \mathbf{R}_{0}^{d}\rightarrow \mathbf{R%
}^{m}$ be a $\mathcal{P\times B}\left( \mathbf{R}_{0}^{d}\right) $%
-measurable function, $H:=H_{r}\left( y\right) ,r\in \lbrack 0,1),y\in 
\mathbf{R}^{d},$ such that for any $T\in \lbrack 0,1)$ a.s.,%
\begin{equation*}
\int_{0}^{T}\int \left\vert H_{r}\left( y\right) \right\vert \rho \left(
y\right) \frac{dydr}{\left\vert y\right\vert ^{d+\alpha }}<\infty ,
\end{equation*}%
Then for each $p\geq 1$ there is $C=C\left( p\right) $ such that for any
stopping time $\tau $, 
\begin{align*}
\mathbf{E}\left[ \sup_{t\leq \tau }\left\vert \int_{0}^{t}\int
H_{r}(y)N(dr,dy)\right\vert ^{p}\right] & \leq C\mathbf{E}\left[
\int_{0}^{\tau }\int \left\vert H_{r}(y)\right\vert ^{p}\rho \left( y\right) 
\frac{dy}{\left\vert y\right\vert ^{d+\alpha }}dr\ \right] \\
& \quad +C\mathbf{E}\left[ \left( \int_{0}^{\tau }\int \left\vert
H_{r}(y)\right\vert \rho \left( y\right) \frac{dy}{\left\vert y\right\vert
^{d+\alpha }}dr\right) ^{p}\right] .
\end{align*}
\end{lemma}

\begin{remark}
\label{re3}Let $H:[0,1)\times \Omega \times \mathbf{R}_{0}^{d}\rightarrow 
\mathbf{R}^{m}$ be a $\mathcal{P\times B}\left( \mathbf{R}_{0}^{d}\right) $%
-measurable function, $H:=H_{r}\left( y\right) ,r\in \lbrack 0,1),y\in 
\mathbf{R}^{d},$ such that for any $T\in \lbrack 0,1)$ a.s.,%
\begin{equation*}
\int_{0}^{T}\int \left\vert H_{r}\left( y\right) \right\vert \rho \left(
y\right) \frac{dydr}{\left\vert y\right\vert ^{d+\alpha }}<\infty .
\end{equation*}%
(i) Since $N\left( dr,dy\right) $-stochastic integral is a sum$,$ a.s. for
every $p\in \left( 0,1\right) ,t\in \lbrack 0,1)$,%
\begin{equation*}
\left\vert \int_{0}^{t}\int H_{r}\left( y\right) N\left( dr,dy\right)
\right\vert ^{p}\leq \int_{0}^{t}\int \left\vert H_{r}\left( y\right)
\right\vert ^{p}N\left( dr,dy\right) .
\end{equation*}%
Hence for any stopping time $\tau ,$%
\begin{equation*}
\mathbf{E}\left[ \sup_{t\leq \tau }\left\vert \int_{0}^{t}\int H_{r}\left(
y\right) N\left( dr,dy\right) \right\vert ^{p}\right] \leq \mathbf{E}%
\int_{0}^{\tau }\int \left\vert H_{r}\left( y\right) \right\vert ^{p}\rho
\left( y\right) \frac{dydr}{\left\vert y\right\vert ^{d+\alpha }}.
\end{equation*}%
(ii) On the other hand, $Z_{t}=\left\vert \int_{0}^{t}\int H_{r}\left(
y\right) N\left( dr,dy\right) \right\vert ,t\in \lbrack 0,1),$ is obviously
dominated by 
\begin{equation*}
A_{t}=\int_{0}^{t}\int |H_{r}\left( y\right) |\rho \left( y\right) \frac{dydr%
}{\left\vert y\right\vert ^{d+\alpha }},t\in \lbrack 0,1).
\end{equation*}%
By Lemma \ref{lem3}, for each $p\in \left( 0,1\right) $, there is $C=C\left(
p\right) >0$ so that%
\begin{equation*}
\mathbf{E}\left[ \sup_{t\leq \tau }\left\vert \int_{0}^{t}\int H_{r}\left(
y\right) N\left( dr,dy\right) \right\vert ^{p}\right] \leq C\mathbf{E}\left[
\left( \int_{0}^{\tau }\int \left\vert H_{r}\left( y\right) \right\vert \rho
\left( y\right) \frac{dydr}{\left\vert y\right\vert ^{d+\alpha }}\right) ^{p}%
\right]
\end{equation*}
\end{remark}

We use the following simple statement about derivation of a global estimate
from a local one.

\begin{lemma}
$\,$ \label{gle}Let $Z_{t},t\in \lbrack 0,1],$ be a nonnegative c\`{a}dl\`{a}%
g stochastic process, $Z_{0}=0$ and $p>0$. Assume there is $\delta \in
\left( 0,1\right) $ and $N,L>0$ such that for any $0\leq S\leq T<1~$with $%
\left\vert T-S\right\vert \leq \delta ,$ we have%
\begin{equation*}
\mathbf{E}\left[ \sup_{S\leq t\leq T}Z_{t}^{p}\right] \leq N[\mathbf{E}\left[
Z_{S}^{p}\right] +L].
\end{equation*}%
Then there is $C=C(\delta ,L,N)$ so that%
\begin{equation*}
\mathbf{E}\left[ \sup_{0\leq t\leq 1}Z_{t}^{p}\right] \leq CL.
\end{equation*}
\end{lemma}

\begin{proof}
We partition $[0,1]$ into $N_{0}$ subintervals of length $N_{0}^{-1}\leq
\delta $. Let $S_{k}=k/N_{0},k=0,\ldots ,N_{0}$, and%
\begin{equation*}
A_{k}=\mathbf{E}\left[ \sup_{S_{k-1}\leq t\leq S_{k}}Z_{t}^{p}\right]
,k=1,\ldots ,N_{0}.
\end{equation*}%
then, 
\begin{eqnarray*}
A_{k} &\leq &NA_{k-1}+NL,k=2,\ldots ,N_{0}, \\
A_{1} &\leq &NL,
\end{eqnarray*}
and then%
\begin{equation*}
A_{k}\leq \left( N^{k}+\ldots +N\right) L=C_{k}L,k=1,\ldots ,N_{0}.
\end{equation*}%
Therefore, 
\begin{equation*}
\mathbf{E}\left[ \sup_{0\leq t\leq 1}Z_{t}^{p}\right] \leq (C_{1}+\ldots
+C_{N_{0}})L.
\end{equation*}
\end{proof}

\end{document}